\documentclass[reqno,final,a4paper]{amsart}
\usepackage[utf8]{inputenc}
\usepackage[notref,notcite]{showkeys} 
\usepackage[shortlabels]{enumitem}
\usepackage[english]{babel}
\usepackage{xcolor}
\usepackage[T1]{fontenc}
\usepackage[normalem]{ulem}

\makeatletter
\def\paragraph{\@startsection{paragraph}{4}%
  \z@{3pt}{-\fontdimen2\font}%
  {\normalfont\bfseries}}
\makeatother

\usepackage{amssymb}
\usepackage{amsfonts}
\usepackage{fullpage}
\usepackage{lmodern}
\usepackage{textcomp}
\usepackage{mathtools, bm}
\usepackage{amssymb, bm}
\usepackage{amsmath}
\usepackage[unq]{unique}
\usepackage{comment}
\usepackage[breaklinks]{hyperref}

\usepackage[numbers]{natbib}
\usepackage{graphicx}

\usepackage{amsthm}
\usepackage{tikz}
\usepackage{caption}
\usepackage{subcaption}
\usepackage{cleveref}
\newtheorem{theorem}{Theorem}[section]

\newtheorem{lemma}[theorem]{Lemma}
\newtheorem{proposition}[theorem]{Proposition}
\newtheorem{corollary}[theorem]{Corollary}

\newtheorem{Claim}[theorem]{Claim}
\newtheorem{claim}[theorem]{Claim}

\theoremstyle{definition}

\newtheorem{conjecture}[theorem]{Conjecture}
\newtheorem*{conjecture*}{Conjecture}

\newtheorem{definition}[theorem]{Definition}

\newcommand{\eps}{\varepsilon}
\newcommand{\e}{\mathrm{e}}
\newcommand{\dtv}{d_{\mathrm{TV}}}

\newcommand{\cC}{\mathcal{C}}

\newcommand{\cF}{\mathcal{F}}
\newcommand{\cG}{\mathcal{G}}

\DeclareMathOperator{\Prob}{\mathbb{P}}

\undef{\abs}
\DeclarePairedDelimiterX{\abs}[1]
  {\lvert}{\rvert}{\ifblank{#1}{\,\cdot\,}{#1}}
\undef{\norm}
\DeclarePairedDelimiterX{\norm}[1]
  {\lVert}{\rVert}{\ifblank{#1}{\,\cdot\,}{#1}}

\title{Monotonicity and decompositions of random regular graphs}

\author[Hollom]{Lawrence Hollom}

\address{Department of Pure Mathematics and Mathematical Statistics, University of Cambridge, Cambridge, CB3 0WA, United Kingdom}
\email{lh569@cam.ac.uk}

\author[Lichev]{Lyuben Lichev}

\address{Institute of Statistics and Mathematical Methods in Economics, TU Wien, Wiedner Hauptstra\ss e 8-10, A-1040 Vienna, Austria}
\email{lyuben.lichev@tuwien.ac.at}

\author[Mond]{Adva Mond}

\address{Department of Mathematics, King’s College London, Strand, London, WC2R 2LS, United Kingdom}
\email{adva.mond@kcl.ac.uk}

\author[Portier]{Julien Portier}

\address{Ecole Polytechnique Federale de Lausanne (EPFL), CH-1015 Lausanne,
Switzerland}
\email{julien.portier@epfl.ch}

\author[Wang]{Yiting Wang}

\address{Institute of Science and Technology Austria (ISTA), 3400 Klosterneuburg, Austria}
\email{yiting.wang@ist.ac.at}

\thanks{Hollom was supported by the Internal Graduate Studentship of Trinity College, Cambridge. Lichev was supported by the Austrian Science Fund (FWF) grant No.~10.55776/ESP624. Mond was supported by UK Research and Innovation grant MR/W007320/2.  Wang was supported by the ERC Starting Grant ``RANDSTRUCT'' No.~101076777. For open access purposes, the authors have applied a CC BY public copyright license to any author accepted manuscript version arising from this submission.}

\begin{document}

\begin{abstract}
In this work we establish several monotonicity and decomposition results in the framework of random regular graphs.
Among other results, we show that, for a wide range of parameters $d_1 \leq d_2$, there exists a coupling of $G(n,d_1)$ and $G(n,d_2)$ satisfying that $G(n,d_1) \subseteq G(n,d_2)$ with high probability, confirming a conjecture of Gao, Isaev and McKay in a new regime.
Our contributions include new tools for analysing contiguity and total variation distance between random regular graph models, a novel procedure for generating unions of random edge-disjoint perfect matchings, and refined estimates of Gao’s bounds on the number of perfect matchings in random regular graphs.
In addition, we make progress towards another conjecture of Isaev, McKay, Southwell and Zhukovskii.
\end{abstract}

\maketitle

\section{Introduction}

The theory of random graphs has seen a number of fascinating advancements over the last few decades.
Part of the effort has been focused on understanding the relations between different models or instances of the same model.
The simplest example of such a relation is illustrated by the widely used technique of \emph{sprinkling} in Erd\H{o}s-R\'enyi graphs or, more generally, in percolation on graphs and lattices.
The technique exploits the simple observation that, given a graph $G$, the $p$-percolated graph $G_p$ can be written as a union of independent random subgraphs of $G$ percolated with probabilities $p_1$ and $p_2$ for $p,p_1,p_2\in [0,1]$ satisfying $1-p=(1-p_1)(1-p_2)$.
While sprinkling has immensely advanced the field of percolation, its applicability is limited to models where edges appear independently of each other.

The efficiency of powerful percolation techniques has motivated the need for establishing connections with other random graph models where methods like sprinkling are not directly available.
A cornerstone work in this direction is due to Kim and Vu~\cite{KV04} who showed that the Erd\H{o}s-R\'enyi random graph $G\sim G(n,p)$ and the random regular graph $H\sim G(n,d)$ with $\omega(\log n) = np\approx d\le n^{1/3-o(1)}$ can be coupled so that, roughly~speaking,
\begin{itemize}
    \item if $np$ is slightly smaller than $d$, then $G\subseteq H$, and
    \item if $np$ is slightly larger than $d$, then $H\setminus G$ contains very few edges.
\end{itemize}
Moreover, Kim and Vu famously conjectured that the range of values where the above relations hold can be extended to $np\approx d \le (1-\eps)n$,
and the second statement can be improved to an exact inclusion of $H$ in $G$.
The significance of this conjecture in the framework of random regular graphs is indisputable: indeed, while switching techniques for computing probabilities in dense random regular graphs and, more generally, random graphs with fixed degrees, exist~\cite{GO23} (see also~\cite{FJP22,JP18,JPRR18,KSVW01,LMP22,McK81,MW91,Wor99} for related applications of switchings), the technical proofs of many related results often lack the elegance of their Erd\H{o}s-R\'enyi counterparts.
The conjecture has seen significant developments in the last twenty years~\cite{DFRS17,GIM22,KRRS23} and is currently confirmed for the range $np\approx d\ge (\log n)^4$ due to Gao, Isaev, and McKay~\cite{GIM20+}. 

While the Kim--Vu conjecture serves as a bridge between Erd\H{o}s-R\'enyi graphs and random regular graphs, comparing random regular graphs of different degree has also attracted attention~\cite{IMcKSZ23,Jan95,MRRW97,Wor99}.
Note that monotonicity for random regular graphs with different degrees (with respect to inclusion) is related to the Kim--Vu conjecture, yet it does not follow from it in any straightforward way.
At the same time, establishing monotone couplings remains highly desirable. 
Such monotonicity for random regular graphs was conjectured by Gao, Isaev and McKay~\cite{GIM22}.

\begin{conjecture}[Conjecture~1.2 in~\cite{GIM22}]
\label{conj:GIM-Inclusion}
Consider integers $d_1\le d_2$ with $d_1,d_2\in [1,n-1]$, such that $(d_1,d_2)\neq (1,2)$, $(d_1,d_2)\neq (n-3,n-2)$ and each of $d_1n$ and $d_2n$ is even. Then, there exists a coupling of $G_1 \sim G(n,d_1)$ and $G_2 \sim G(n,d_2)$ such that $\mathbb{P}(G_1 \subseteq G_2)=1-o(1)$.
\end{conjecture}

\Cref{conj:GIM-Inclusion} is known to hold in a few cases~\cite{Gao23,GIM20+,IMcKSZ23}:
\begin{itemize}
    \setlength\itemsep{0.3em}
    \item $d_1 = 1$ and $d_2\in [3,n-1]$,
    \item $\omega((\log n)^7) = d_1\le n^{1/7}/\log n$ and $d_2\in [d_1,n-1]$,
    \item $\omega(1) = d_1\le n^{1/3}$ and $d_2 = d_1+1$,
    \item $(\log n)^4\le c d_1\le d_2\le n-(\log n)^4$ for some constant $c > 1$,
    \item $\min\{d_1,d_2-d_1\} = \omega(n/\log n)$.
\end{itemize}
A closer look at the listed cases reveals that, while the monotonicity properties of random regular graphs of large polylogarithmic degrees are relatively well understood, graphs with slowly growing degrees remain elusive.
Our first result bridges this gap and confirms \Cref{conj:GIM-Inclusion} in the case when $n$ is even and $d_1$ is not too large.

\begin{theorem}
\label{thm:Coupling-Inclusion}
Consider $n$ even, $d_1\in [3,n^{1/7}/\log n]$ and $d_2\in [d_1,n-1]$ with $d_2 = \omega(1)$.
Then there exists a coupling of $G_1 \sim G(n,d_1)$ and $G_{2} \sim G(n,d_2)$ such that $\mathbb{P}(G_1 \subseteq G_2)=1-o(1)$.
\end{theorem}

\Cref{thm:Coupling-Inclusion} will be derived by combining the second point in the above list with several technical results guaranteeing that, in a certain sense, $G(n,d_2)$ can be decomposed into edge-disjoint $G(n,d_1)$ and $d_2-d_1$ perfect matchings for $3\le d_1\le d_2 \le n^{1/10}$.

While \Cref{conj:GIM-Inclusion} requires that the graphs $G_1$ and $G_2$ are chosen uniformly at random, no restriction is imposed on the graph $G_2\setminus E(G_1)$.
With an eye towards recovering the properties of the sprinkling technique in the framework of random regular graphs, requiring that the complement of $G_1$ behaves as a random $(d_2-d_1)$-regular graph is a natural additional restriction.
More precisely, for $d_1,d_2\in [1,n-1]$ with $d_1\le d_2$, we denote by $G(n,d_1)\oplus G(n,d_2-d_1)$ the random $d_2$-regular graph obtained by sampling independently $G_1\sim G(n,d_1)$ and $G_2'\sim G(n,d_2-d_1)$ conditionally on $G_1,G_2'$ being edge-disjoint, and taking their union.
Isaev, McKay, Southwell and Zhukovskii~\cite{IMcKSZ23} conjectured that the distribution of $G(n,d_1)\oplus G(n,d_2-d_1)$ is essentially the one of $G(n,d_2)$ as long as $d_2$ tends to infinity.

\begin{conjecture}[Conjecture 1.2 in \cite{IMcKSZ23}]\label{conj:IMSZ-coupling-sum}
For $d_1,d_2\in [1,n-1]$ with $d_1\le d_2 =\omega(1)$, there exists a coupling of $G_{\oplus} \sim G(n,d_1) \oplus G(n,d_2-d_1)$ and $G\sim G(n,d_2)$ such that $\mathbb{P}(G_{\oplus} = G)=1-o(1)$.
\end{conjecture}

The authors of~\cite{IMcKSZ23} verified \Cref{conj:IMSZ-coupling-sum} in several cases, namely:
\begin{itemize}
\setlength\itemsep{0.3em}
    \item $d_1 = 1$, $d_2=\omega(1)$ and $d_2 = o(n^{1/3})$,
    \item $\min\{d_1,d_2-d_1,n-d_2\} = \omega(n/\log n)$,
    \item $d_1^2 (n/\min\{d_2-d_1,n-d_2+d_1\})^{2d_1}\le (\log n)/3$.
\end{itemize}
They also observed that the assumption $d_2=\omega(1)$ is actually necessary. We contribute to the understanding of \Cref{conj:IMSZ-coupling-sum} by verifying that, for large  $d_2 \le n^{1/10}$, the distributions of $G(n,d_1)\oplus G(n,d_2-d_1)$, with $d_1$ far from $1$ and $d_2$, are asymptotically the same.

\begin{theorem}
\label{thm:Coupling-Sprinkling-2-Random}
Consider $n$ even, $d_2=\omega(1)$ and $d_2 = O(n^{1/10})$. 
Fix $d_1,d_1'$ with
\[\min\{d_1,d_1',d_2-d_1,d_2-d_1'\}=\omega(1).\]
Then, there exists a coupling of $G_{\oplus} \sim G(n,d_1) \oplus G(n,d_2-d_1)$ and $G_{\oplus}' \sim G(n,d'_1) \oplus G(n,d_2-d_1')$ such that $\mathbb{P}(G_{\oplus} = G_{\oplus}')=1-o(1)$.
\end{theorem}

Simply put, while \Cref{thm:Coupling-Sprinkling-2-Random} does not claim that $G(n,d_2)$ can be decomposed as a disjoint union of $G(n,d_1)$ and $G(n,d_2-d_1)$, it confirms that this disjoint union has asymptotically the same distribution as long as each of $d_1$ and $d_2-d_1$ is large.

A classic result due to Molloy, Robalewska, Robinson and Wormald~\cite{MRRW97} ensures that, when $d = O(1)$ and $d_1+\dotsb+d_k = d$ with $(d,d_1,d_2)\neq (2,1,1)$, the properties holding with high probability in the model $G(n,d_1)\oplus \dotsb \oplus G(n,d_k)$ also hold with high probability in the model $G(n,d)$ and vice versa. 
Using this result to decompose a random $d$-regular graph on an even number of vertices into a disjoint union of perfect matchings and Hamilton cycles has sometimes been instrumental for the analysis of random regular graphs, see~\cite{BWZ02,KW01,Wor99}.
In a more structural setting, Ferber and Jain~\cite{FB18} showed that pseudorandom $d$-regular graphs on an even number of vertices can be decomposed into $d$ edge-disjoint perfect matchings, that is, such graphs have a 1-factorisation.
While the latter result holds in vast generality and, in particular, implies that $G(n,d)$ has a 1-factorisation with high probability, the probabilistic properties of the exhibited 1-factorisation remain unclear.
In this work, at the price of restricting the range of values for $d$, we show that the model $G(n,1)\oplus \dotsb \oplus G(n,1)$, with $G(n,1)$ repeated $d$ times, serves as a good approximation of the model of random $d$-regular graphs in the following sense.

\begin{theorem}
\label{thm:Contiguity-Disjoint-Matchings}
Consider $n$ even and $d\in [3,n^{1/10}]$.
Fix a sequence $(A_n)_{n\in 2\mathbb N}$ where $A_n$ is a set of $d$-regular graphs on $n$ vertices.
If $G(n,1)\oplus \dotsb \oplus G(n,1)\in A_n$ holds with high probability, then $G(n,d)\in A_n$ holds with high probability as well.
\end{theorem}

We remark that, with some more effort, the range of values for $d$ in \Cref{thm:Coupling-Sprinkling-2-Random,thm:Contiguity-Disjoint-Matchings} may be slightly improved but the current estimates on the variance of the number of perfect matchings in random regular graphs (see \cite[Theorem~10]{Gao23} needed for our \Cref{prop:Estimate-E[(Y-Y*)^2]}) do not allow us to go beyond $d\le n^{1/7-\eps}$.
We restrict our analysis to the range $d\in [3,n^{1/10}]$ for the sake of simplicity.
In general, we believe that \Cref{thm:Contiguity-Disjoint-Matchings} as well as its reciprocal implication should hold for $d\in [3,n-3]$.

\begin{conjecture}
Consider $n$ even and $d\in [3,n-3]$.
Fix a sequence $(A_n)_{n\in 2\mathbb N}$ where $A_n$ is a set of $d$-regular graphs on $n$ vertices.
Then, $G(n,1)\oplus \dotsb \oplus G(n,1)\in A_n$ holds with high probability if and only if $G(n,d)\in A_n$ holds with high probability.
\end{conjecture}

\paragraph{Outline of the main ideas.}

The proofs of \Cref{thm:Coupling-Sprinkling-2-Random,thm:Contiguity-Disjoint-Matchings} are based on two results: \Cref{cor:cont_Gnd} and \Cref{cor:dTV}. 
\Cref{cor:cont_Gnd} specialises a more general technical result, \Cref{prop:general}, which establishes that, for measures $\eta_n$, $\tilde\eta_n$ and $\tau_n$ on the space of $n$-vertex graphs, 
if the measures $\eta_n, \tilde\eta_n$ are close (in total variation distance, or contiguous) and $\tau_n$ satisfies a weak technical assumption, then the measures $\eta_n\oplus \tau_n, \tilde\eta_n\oplus \tau_n$ are also close in the corresponding sense.
\Cref{cor:cont_Gnd} then specialises this observation to the setting where $\tau_n$ has distribution $G(n,d_1)\oplus G(n,d_2)\oplus\dotsb\oplus G(n,d_k)$ for a wide range of parameters $k$ and $d_1,\ldots,d_k$.

\Cref{cor:dTV} bounds the total variation distance between $G(n,d)\oplus G(n,1)$ and $G(n,d+1)$ whenever $d\le n^{1/10}$.
We observe that a weaker version of \Cref{cor:dTV} can be derived directly from a work of Gao~\cite{Gao23}. 
However, the explicit error bounds one can extract from Gao's proof are inconveniently large, and do not allow us to iterate the result to conclude that $G(n,d_1)\oplus G(n,1)\oplus\dotsb\oplus G(n,1)$ and $G(n,d+1)$ are close in total variation distance for $d_1 = o(d)$.
To overcome this inconvenience, we build on Gao's approach, which itself relies on the orthogonal decomposition and projection method introduced by Janson~\cite{Jan94}.
An application of this technique allows us to approximate well the number of perfect matchings via a linear regression involving the number of triangles, which is significantly easier to analyse due to the locality of the triangle count.
Computing the third and the fourth moments of the number of triangles in the random $d$-regular graph (see \Cref{lem:estimate-2}) allows us to sharpen concentration estimates coming from the second moment method.
In turn, this enables us to refine Gao's bounds through the orthogonal decomposition and projection method.
Equipped with our knowledge about perfect matchings in random regular graphs, we analyse the total variation distance between $G(n,d)\oplus G(n,1)$ and $G(n,d+1)$ by using a result of McKay~\cite{McK85} (see~\Cref{thm:McKay}) about counting sparse graphs with given degree sequence in dense host graphs, and a convenient consequence of Strassen's theorem observed by Isaev, McKay, Southwell and Zhukovskii~\cite{IMcKSZ23} (see~\Cref{thm:strassen}). 

Finally, while \Cref{thm:Coupling-Inclusion} can be derived from \Cref{cor:cont_Gnd} and \Cref{cor:dTV} when $d_1=\omega(1)$, the two results are insufficient when $d_1 = O(1)$. 
Indeed, when $d_1 = O(1)$, the bound on the total variation distance between $G(n,d)\oplus G(n,1)$ and $G(n,d+1)$ given by \Cref{cor:dTV} is a positive constant and, therefore, does not provide an inclusion-wise monotone coupling with high probability.
In this case, for all $k\ge 1$, our approach relies on developing an alternative sampling procedure (see~\Cref{def:procedure-sample}) to approximately generate $G(n,1)\oplus\ldots\oplus G(n,1)$, with $G(n,1)$ repeated $kd_1$ times, in blocks of $d_1$ perfect matchings.  
To do so, we rely on a random overlay idea: 
we show that, for $m = O(1)$ unlabelled regular graphs with degrees of order $O(1)$, by attributing the same vertex labels to them uniformly at random and independently, the resulting labelled graphs are edge-disjoint with probability bounded away from $0$ and depending only on $m$ and the degree of each graph (see~\Cref{lem:disjoint}).
Then, we show that, for every $\eps > 0$, one can choose a large constant $k= k(\eps)\ge 1$ and couple $G(n,d_1)$ and the graph generated by the alternative procedure so that $G(n,d_1)$ coincides with a graph corresponding to one of the $k$ blocks with probability at least $1-\eps$ (see~\Cref{lem:coupleXY,lem:complete_couple}).
We note that, with very minor modifications, \Cref{lem:coupleXY} provides a general coupling tool of measures $\mu$ and $\nu\times \ldots\times \nu$ where $\mu,\nu$ are absolutely continuous measures on the same space, and may be of independent interest.

\paragraph{Organisation of the paper.} In \Cref{sec:prelims}, we state and establish some preliminary results.
\Cref{sec:technical} is dedicated to the proof of \Cref{prop:PM-counting} which is used to derive \Cref{thm:Coupling-Sprinkling-2-Random,thm:Contiguity-Disjoint-Matchings} in \Cref{sec:proofs45}.
Section~\ref{sec:coupling} contains some coupling results instrumental in the proof of \Cref{thm:Coupling-Inclusion} in \Cref{sec:proofs}.
We conclude with a brief discussion in \Cref{sec:conclusion}.

\section{Notation and preliminaries}\label{sec:prelims}
We first recall (mostly classic) notation and terminology used in the sequel.

\subsection{Notation and terminology}\label{sec:notations}
\paragraph{General notation.} For a positive integer $n$, we denote $[n] = \{1,\ldots,n\}$. For positive integers $m\le n$, we also write $\binom{[n]}{m}$ for all subsets of $[n]$ of size $m$.
For real numbers $a,b,c$ with $c > 0$, we write $a = b\pm c$ to mean that $a\in [b-c,b+c]$. 
Floor and ceiling notation is omitted when rounding is insignificant for the argument.

Given a graph $G$, we write $V(G)$ for its vertex set and $E(G)$ for its edge set.

For integers $n\ge 1$ and $d\in [1,n-1]$, we denote by $\cG_d = \cG_d(n)$ the family of $d$-regular graphs on $n$ vertices.
We denote by $\mu_{d,n}$ the uniform probability distribution on the family $\cG_d(n)$, but often omit the index $n$ and write $\mu_d$ for better readability.
For integers $k\ge 1$ and $d_1,\ldots,d_k\ge 1$ with $d_1+\dotsb+d_k\le n-1$, we denote by $G(n,d_1)\oplus \dotsb \oplus G(n,d_k)$ the union of independent copies of $G(n,d_1),\ldots,G(n,d_k)$ on the same vertex set conditionally on being edge-disjoint.
We denote by $\nu_{d,n}$ the probability distribution on $\cG_d(n)$ such that $\nu_{d,n}(G)$ is proportional to the number of 1-factorisations of $G$.
Again, we often write $\nu_d$ for better readability. 
Given two measures $\mu$ and $\nu$ on the same family of graphs $\mathcal G$, $\mu\oplus\nu$ is a measure on $\cG$ such that, for every $G\in \mathcal G$,
$(\mu\oplus \nu)(G)$ is proportional to 
\[
\sum_{\substack{G_1,G_2 \in \mathcal{G}\\ G_1\cup G_2 = G\\ G_1\cap G_2 = \emptyset}}\mu(G_1)\nu(G_2)\,,
\]
where $G_1\cap G_2 = \emptyset$ means that the edge sets of $G_1$ and $G_2$ are disjoint.

We work with a measurable space $(\Omega,\mathcal F)$ where $\Omega$ is always a finite or countable set. By default, we assume $\mathcal F$ to be the family of all subsets of $\Omega$ and omit it from the notation.
A sequence of events $(E_n)_{n\ge 1}$ is said to hold \emph{with high probability} or \emph{whp} if $\mathbb P(E_n)\to 1$ as $n\to \infty$.

We use standard asymptotic notations.
For functions $f=f(n)$ and $g=g(n)$ (with $g$ non-negative), we write $f=O(g)$ to say that there is a constant $C$ such that $|f(n)|\le Cg(n)$, and we write $f=\Omega(g)$ to say that there is a constant $c>0$ such that $f(n)\ge cg(n)$ for sufficiently large $n$.
We write $f=\Theta(g)$ to say that $f=O(g)$ and $f=\Omega(g)$, and we also write $f=o(g)$ or $g = \omega(f)$ to say that $f(n)/g(n)\to 0$ as $n\to\infty$.

\paragraph{Coupling.} For two (probability) distributions $\mu_1$ and $\mu_2$ on the measurable spaces $\Omega_1$ and $\Omega_2$, respectively, a \emph{coupling} of $\mu_1$ and $\mu_2$ is a (probability) distribution on the space $\Omega_1\times \Omega_2, $ with first marginal $\mu_1$ and second marginal $\mu_2$.

\paragraph{Total variation distance.} For two probability distributions $\mu_1$ and $\mu_2$ defined on a common measurable space $\Omega$, the \emph{total variation distance} between $\mu_1$ and $\mu_2$ is defined as
\begin{align*}
\dtv(\mu_1,\mu_2)
&= \sup_{A} |\mu_1(A)-\mu_2(A)|,
\end{align*} 
where the supremum is taken over all subsets of $\Omega$.
Moreover, since we assume the sample space $\Omega$ to be finite or countable, the above definition is equivalent to
\begin{equation}\label{eq:dTV}
\dtv(\mu_1,\mu_2) = \frac{1}{2}\sum_{x\in \Omega} |\mu_1(x)-\mu_2(x)|.
\end{equation} 

\paragraph{Contiguity.} Fix a sequence of measurable spaces $(\Omega_n)_{n\ge 1}$. For every $n\ge 1$, consider two probability measures $\mu_n$ and $\nu_n$ on the space $\Omega_n$.
We say that $(\mu_n)_{n\ge 1}$ and $(\nu_n)_{n\ge 1}$ are \emph{contiguous} if, for every sequence of events $(A_n)_{n\ge 1}$, 
we have that
\[\mu_n(A_n)\xrightarrow[n\to \infty]{} 1\iff \nu_{n}(A_{n})\xrightarrow[n \to \infty]{} 1\,.\]
For convenience, we will often abuse terminology and say that $\mu_n$ and $\nu_n$ are contiguous. 
It is easy to see that $\mu_n$ and $\nu_n$ are contiguous if $\dtv(\mu_n, \nu_n) = o(1)$.
However, the converse does not hold in general, as contiguity only concerns sequences of events which hold whp.
For instance, by setting $\Omega_n = \{0,1\}$, $\mu_n = \mathrm{Ber}(1/3)$ and $\nu_n = \mathrm{Ber}(2/3)$ for all $n\ge 1$, the measures $\mu_n$ and $\nu_n$ are contiguous but $\dtv(\mu_n, \nu_n) = 1/3$ for all $n\ge 1$. 

\subsection{Preliminary results}\label{sec:prelim_results} 

\subsubsection{\emph{\textbf{Helpful results from the literature}}}
We begin this section with a theorem of McKay~\cite{McK85} determining the number of graphs with a given (sparse) degree sequence $g_1,\ldots,g_n$ in the complement of a (sparse) graph $X$.

\begin{theorem}[Theorem~4.6 in~\cite{McK85}]\label{thm:McKay}
Fix $\varepsilon\in (0,2/3)$, a graph $X$ with vertex set $[n]$ and degree sequence $x_1,\ldots,x_n$, and integers $g_1,\ldots,g_n\in [0,n-1]$ with even sum. Define
\[e(G) = \frac{1}{2}\sum_{i=1}^n g_i,\quad \lambda = \frac{1}{4e(G)} \sum_{i=1}^n g_i(g_i-1)\quad \text{and}\quad \mu = \frac{1}{2e(G)} \sum_{ij\in E(X)} g_ig_j.\]
Suppose that 
\[\hat\Delta := 2+ \max_{i\in [n]} g_i\bigg(\frac{3}{2}\max_{i\in [n]} g_i + \max_{i\in [n]} x_i + 1\bigg)\le \eps\sum_{i=1}^n g_i.\]
Then, the number of graphs with degree sequence $g_1,\ldots,g_n$ on the vertex set $[n]$ which are edge-disjoint from $X$ is equal to
\[\frac{(2e(G))!}{e(G)! \, 2^{e(G)}\prod_{i=1}^n g_i!} \exp\left(- \lambda - \lambda^2 - \mu - O\bigg(\frac{\hat\Delta^2}{e(G)}\bigg)\right),\]
where the constant in the $O(\cdot)$ term is uniform over the choice of all parameters.
\end{theorem}

The following preliminary result concerns contiguity between models of random regular graphs of bounded degree.
It is a particular case of Corollary~4.17 in~\cite{Wor99}; see also~\cite{Jan95,MRRW97}.

\begin{theorem}[see Corollary~4.17 in~\cite{Wor99}]\label{thm:const}
Fix $k\ge 2$ and $d_1,\ldots,d_k\ge 1$ with $d_1+\dotsb+d_k=d\ge 3$.
For even $n$, the models $G(n,d)$ and $G(n,d_1)\oplus \dotsb\oplus G(n,d_k)$ are contiguous.
\end{theorem}

Next, we state a coupling result from~\cite{IMcKSZ23} related to Strassen's theorem, see also \cite[Proposition 6]{Kop24}.

\begin{theorem}[Corollary 2.2 from~\cite{IMcKSZ23}]\label{thm:strassen}
Fix $\delta,\varepsilon\in [0,1]$ and a bipartite graph $G$ with parts $S,T$.
Suppose that $S$ contains at least $(1-\delta)|S|$ vertices of degree at least $(1-\eps)|E(G)|/|S|$, and similarly $T$ contains at least $(1-\delta)|T|$ vertices of degree at least $(1-\eps)|E(G)|/|T|$.
Then, there is a coupling $(X,Z)$ with $X$ uniformly distributed on $S$ and $Z$ uniformly distributed on $T$ such that $\mathbb P(XZ\notin G)\le 2\delta+\eps/(1-\eps)$.
\end{theorem}

The last lemma in this section guarantees that two measures $\mu,\nu$ on the same measurable space can be coupled in such a way that their outputs coincide with probability $1-\dtv(\mu,\nu)$, and ensures an additional independence property.
The version stated below (in particular, parts (ii) and (iii)) are slightly different from the original Theorem~2.12 in~\cite{Den12}: part (ii) differs due to the scaling of $\dtv$ in~\cite{Den12} while part (iii) follows from the proof of Theorem~2.12 in~\cite{Den12} rather than the statement itself.
\begin{lemma}[see Theorem 2.12 in~\cite{Den12} and its proof]\label{lem:max_coupling}
For any two probability distributions $\mu$ and $\nu$ on a finite or countable space $\Omega$, there exists a coupling $(X,Y)$ such that the following hold simultaneously:
\begin{itemize}
\setlength\itemsep{0.3em}
    \item[\emph{(i)}] $X$ has distribution $\mu$ and $Y$ has distribution $\nu$,
    \item[\emph{(ii)}] $X=Y$ with probability $1-\dtv(\mu,\nu)$,
    \item[\emph{(iii)}] $X$ and $Y$ are independent conditionally on the event $X\neq Y$, provided that the latter event holds with positive probability. In this case, for every $x\in \Omega$, the probability that $X=x$ conditionally on $X\neq Y$ is proportional to $\mu(x) - \min(\mu(x),\nu(x))$.
\end{itemize}
\end{lemma}

\subsubsection{\emph{\textbf{Other preliminary results.}}} 
We start with a simple observation used several times throughout the proof, often without explicit mention.

\begin{lemma}[$\oplus$ is associative for measures]\label{lem:assoc}
Fix three probability distributions $\mu, \nu, \eta$ on the same family of graphs $\mathcal G$. Then, $(\mu\oplus \nu)\oplus \eta = \mu\oplus\nu\oplus \eta$.
\end{lemma}
\begin{proof}
For any $G\in \mathcal G$, note that $((\mu\oplus\nu)\oplus \eta)(G)$ is proportional to 
\[
    \sum_{\substack{G_1,G_3\in \mathcal G\\ G_1\cup G_3 = G\\ G_1\cap G_3 = \emptyset}}(\mu\oplus \nu)(G_1)\cdot \eta(G_3) = \sum_{\substack{G_1,G_2,G_3\in \mathcal G\\ G_1\cup G_2\cup G_3 = G\\ G_1,G_2, G_3 \text{ pairwise edge-disjoint}}} \mu(G_1)\nu(G_2)\eta(G_3)\,.
\]
Moreover, $(\mu\oplus \nu\oplus \eta)(G)$ is also proportional to the same quantity. Since each of $(\mu\oplus\nu)\oplus \eta$ and $\mu\oplus\nu\oplus\eta$ is a probability distribution, they must coincide.
\end{proof}
\noindent
An example for an application of \Cref{lem:assoc} is that for every $d_1,\ldots,d_k\ge 1$ which sum to $d$, we can write $\nu_d = \nu_{d_1} \oplus \dotsb \oplus \nu_{d_k}$.

The next proposition quantifies the contiguity relation between two measures and shows that, if two sequences are contiguous, then their total variation distance is bounded away from $1$.

\begin{proposition}\label{prop:cont-dtv}
Fix a sequence of finite or countable measurable spaces $(\Omega_n)_{n\ge 1}$ 
and contiguous probability distributions $\eta_n$ and $\tilde\eta_n$ on $\Omega_n$.
Denote $\eta = (\eta_n)_{n\ge 1}$ and $\tilde\eta = (\tilde\eta_n)_{n\ge 1}$.
Then, each of the following holds.
\begin{enumerate}
\setlength\itemsep{0.3em}
    \item[\emph{(i)}] For every $\delta > 0$, there exist $\gamma = \gamma(\delta, \eta,\tilde\eta)\in (0,\delta)$ and $n_0 = n_0(\delta, \eta, \tilde\eta) \ge 1$ such that, for every $n\ge n_0$ and every event $A_n\subseteq \Omega_n$, $\eta_n(A_n)\le \gamma$ implies $\tilde\eta_n(A_n)\le \delta$ and $\tilde\eta_n(A_n)\le \gamma$ implies $\eta_n(A_n)\le \delta$.
    \item[\emph{(ii)}] There exist an absolute constant $\eps > 0$ and $n_0\in \mathbb N$ such that $\dtv(\eta_n,\tilde{\eta}_n) \leq 1 - \eps$ for all $n\geq n_0$. 
\end{enumerate}
\end{proposition}
\begin{proof}
For (i), we show the first implication; the second one is derived similarly.
Suppose for contradiction that there is $\delta > 0$ and a sequence $1\le n_1 < n_2 < \dotsb$ such that, for every $i\ge 1$, there is an event $E_{n_i}$ such that $\eta_{n_i}(E_{n_i})\le 1/i$ but $\tilde\eta_{n_i}(E_{n_i}) > \delta$.

Define events $(A_n)_{n\in \mathbb N}$ where $A_{n} = \Omega_n$ for $n\notin \{n_1,n_2,\dots\}$ and, for every $i\ge 1$, $A_{n_i} = E_{n_i}^c$. Then, by construction, we have that $\lim_{n\rightarrow \infty}\eta_n(A_n) = 1$, which by contiguity implies $\lim_{n\rightarrow \infty}\tilde{\eta}_n(A_n) = 1$. However, this contradicts the fact that $\limsup_{i\rightarrow \infty}\tilde{\eta}_{n_i}(A_{n_i}) \le 1-\delta$.
Thus, there exists a constant $\gamma_0 > 0$ satisfying (i), and choosing $\gamma = \min(\delta, \gamma_0)$ finishes the proof of (i).

Now, we deduce (ii) from (i). Fix any $\delta \in (0,1)$ and $\gamma, n_0$ as in (i). 
Fix any sequence of events $(A_n)_{n\geq 1}$ and $n\geq n_0$.
We show that $|\eta_n(A_n) - \tilde{\eta}_n(A_n)|$ is bounded away from $1$.
If $\eta_n(A_n) \leq \gamma$, then $\tilde{\eta}_n(A_n)\leq \delta$ and $|\eta_n(A_n) - \tilde{\eta}_n(A_n)| \leq \delta$, which is bounded away from 1. 
Note that a similar inequality holds when $\tilde\eta_n(A_n) \leq \gamma$.
Finally, if $\min(\eta_n(A_n),\tilde\eta_n(A_n))\ge \gamma$, then $|\eta_n(A_n) - \tilde{\eta}_n(A_n)| \leq 1-\gamma$, which is again bounded away from 1, completing the proof. 
\end{proof}

The following lemma shows that contiguity has a convenient fractional reformulation.
\begin{lemma}
\label{lem:GeneralisationContiguity}
Fix a sequence of finite or countable measurable spaces $(\Omega_n)_{n\ge 1}$ 
and contiguous probability measures $\eta_n$ and $\tilde\eta_n$ on $(\Omega_n)_{n\geq 1}$.
Fix a family $\{a_{x,n}:n\ge 1,x\in \Omega_n\}$ of uniformly bounded non-negative real numbers.
Then, 
\[\sum_{x\in \Omega_n} a_{x,n}\, \eta_n(x)\xrightarrow[n\to \infty]{} 0\iff \sum_{x\in \Omega_n} a_{x,n}\, \tilde\eta_n(x)\xrightarrow[n\to \infty]{} 0\,.\]
\end{lemma}

Note that one retrieves the classic definition of contiguity when $a_{x,n} \in \{0,1\}$ for every $n \ge 1$ and $x \in \Omega_n$.

\begin{proof}
We show that convergence on the left-hand side implies convergence on the right-hand side; the inverse implication follows by symmetry.
Suppose that $\sum_{x\in \Omega_n} a_{x,n} \eta_n(x)\to 0$.
Fix $\eps >0$ and, for every $n \ge 1$, define the sets $S_n = \{x\in \Omega_n : a_{x,n} \leq \eps \}$ and $L_n =\{x\in \Omega_n : a_{x,n} > \eps \}$.
Then,
\begin{align*}
\sum_{x\in \Omega_n} a_{x,n}\, \eta_n(x) \ge \sum_{x\in L_n} a_{x,n}\, \eta_n(x) \ge \eps\, \eta_n(L_n).
\end{align*}
Hence, $\eta_n(L_n)\to 0$ and, by contiguity of the measures $\eta_n$ and $\tilde\eta_n$, we have $\tilde\eta_n(L_n)\to 0$.
As a result, for every $\eps > 0$, there is $n_0 \ge 1$ such that, for every $n \geq n_0$,
\begin{align*}
\sum_{x\in \Omega_n} a_{x,n}\, \tilde\eta_n(x) = \sum_{x\in L_n} a_{x,n}\, \tilde\eta_n(x) + \sum_{x\in S_n} a_{x,n}\, \tilde\eta_n(x) \le \max\{a_{x,n}:n\ge 1,x\in \Omega_n\}\, \tilde\eta_n(L_n) + \eps \tilde\eta_n(S_n)\le 2\eps.
\end{align*}
As the latter holds for every $\eps > 0$, this completes the proof.
\end{proof}

The next proposition relates the total variation distance and contiguity of measures with the operation~$\oplus$.
\begin{proposition}\label{prop:general}
For every integer $n\ge 1$, fix probability distributions $\eta_n$, $\tilde\eta_n$ and $\tau_n$ 
on the family $\cF_n$ of labelled $n$-vertex graphs.
For every graph $G\in \cF_n$, define
\begin{equation}\label{eq:beta}
\beta_n(G) = \sum_{H\in \cF_n:\, G\cap H=\emptyset} \tau_n(H).
\end{equation}
Suppose that, for some constant $\gamma > 0$ and every $G,G'\in \cF_n$, we have $\beta_n(G)\le \gamma\beta_n(G')$. Then, each of the following holds. 
\begin{enumerate}
\setlength\itemsep{0.3em}
    \item[\emph{(a)}] $\dtv(\eta_n\oplus\tau_n,\tilde\eta_n\oplus\tau_n)\le 2\gamma\dtv(\eta_n,\tilde\eta_n)$.
    \item[\emph{(b)}] Suppose that the sequences of measures $(\eta_n)_{n\ge 1}$ and $(\tilde\eta_n)_{n\ge 1}$ are contiguous. Then, the sequences $(\eta_n\oplus\tau_n)_{n\ge 1}$ and $(\tilde\eta_n\oplus\tau_n)_{n\ge 1}$ are contiguous as well.
\end{enumerate}
\end{proposition}

Note that, in \Cref{prop:general}, $\cF_n$ is not necessarily the support of $\eta_n$, $\tilde\eta_n$ and $\tau_n$, meaning that some graphs $G\in \cF_n$ can be attributed weight 0 by any of these measures.

\begin{proof}[Proof of \Cref{prop:general}]
To begin with, we focus on the proof of part (a). Denote for brevity $\lambda_n = \eta_n \oplus \tau_n$ and $\tilde\lambda_n= \tilde\eta_n \oplus \tau_n$.
Moreover, for every $S\in \cF_n$, define 
\begin{align*}
\psi_n(S) \coloneqq \sum_{\substack{G,H\in \cF_n\\ G\cup H=S\\ G\cap H = \emptyset}} \eta_n(G) \tau_n(H)\qquad \text{and}\qquad \Psi_n \coloneqq \sum_{S\in \cF_n} \psi_n(S),
\end{align*}
and denote by $\tilde \psi_n(S)$ and $\tilde \Psi_n$ the analogous quantities obtained by replacing $\eta_n$ with $\tilde\eta_n$.
Then, we have
\begin{align*}
\dtv(\lambda_n,\tilde\lambda_n) = \frac{1}{2} \sum_{S\in \cF_n} |\lambda_n(S) -\tilde\lambda_n(S)| = \frac{1}{2} \sum_{S\in \cF_n} \left|\frac{\psi_n(S)}{\Psi_n} - \frac{\tilde\psi_n(S)}{\tilde\Psi_n} \right|.
\end{align*}
By summing over $G$ first, note that $\Psi_n = \sum_{G\in \cF_n} \eta_n(G) \beta_n(G)$, and similar equality holds for 
$\tilde\Psi_n$ and $\tilde\eta_n$.
In particular, by the triangle inequality,
\begin{equation}\label{eq:triangle}
|\Psi_n - \tilde\Psi_n|
\le \sum_{G\in \cF_n} |\eta_n(G)-\tilde\eta_n(G)| \beta_n(G) \le 2\dtv(\eta_n,\tilde\eta_n)\max_{G\in \cF_n}\beta_n(G).
\end{equation}
By using \eqref{eq:triangle} and the triangle inequality again, we obtain that 
\begin{align*}
\dtv(\lambda_n,\tilde\lambda_n) 
&\le \frac{1}{2} \sum_{S\in \cF_n} \left|\frac{\psi_n(S)}{\Psi_n} - \frac{\tilde\psi_n(S)}{\Psi_n} \right| + \frac{1}{2} \sum_{S\in \cF_n} \left|\frac{\tilde\psi_n(S)}{\Psi_n} - \frac{\tilde\psi_n(S)}{\tilde\Psi_n} \right|\\
&\leq \frac{\sum_{G,H\in \cF_n:\, G\cap H = \emptyset} |\eta_n(G)-\tilde\eta_n(G)| \tau_n(H)}{2\Psi_n} + \sum_{S\in \cF_n} \frac{\tilde\psi_n(S)}{2}\frac{|\Psi_n-\tilde\Psi_n|}{\Psi_n\tilde\Psi_n}\\
&= \frac{\sum_{G\in \cF_n} |\eta_n(G)-\tilde\eta_n(G)| \beta_n(G)}{2\Psi_n} + \frac{|\Psi_n-\tilde\Psi_n|}{2\Psi_n}\le \frac{2\dtv(\eta_n,\tilde\eta_n)\max_{G\in \cF_n}\beta_n(G)}{\Psi_n}.
\end{align*}
Finally, by using that $\Psi_n = \sum_{G\in \cF_n} \eta_n(G) \beta_n(G)\ge \min_{G\in \cF_n} \beta_n(G)$ and the assumption on $\beta_n$ in the statement, we obtain that
\[\dtv(\lambda_n,\tilde\lambda_n)\le \frac{2\dtv(\eta_n,\tilde\eta_n)\max_{G\in \cF_n}\beta_n(G)}{\min_{G\in \cF_n} \beta_n(G)}\le 2\gamma \dtv(\eta_n, \tilde\eta_n),\]
which finishes the proof of part (a).

We turn to the proof of part (b). For every $n\ge 1$, fix $A_n\subseteq \cF_n$. 
Then,
\begin{align*}
\lambda_n(A_n) = \frac{\sum_{S\in A_n} \psi_n(S)}{\Psi_n} \qquad\text{and}\qquad\tilde\lambda_n(A_n) = \frac{\sum_{S\in A_n} \tilde\psi_n(S)}{\tilde\Psi_n}.
\end{align*}
Suppose that $\lambda_n(A_n)\to 0$.
For every $n\ge 1$ and $G\in \cF_n$, denote
\[a_{G,n}\coloneqq \frac{\sum_{H\in \cF_n:\, G\cup H\in A_n,\, G\cap H = \emptyset} \tau_n(H)}{\max_{G\in \cF_n} \beta_n(G)}\in [0,1].\]
Using that $\Psi_n\le \max_{G\in \cF_n} \beta_n(G)$, we obtain that
\begin{align*}
\lambda_n(A_n) = \frac{\sum_{S\in A_n} \psi_n(S)}{\Psi_n} \ge \sum_{G\in \cF_n} \frac{\sum_{H\in \cF_n:\, G\cup H\in A_n,\, G\cap H = \emptyset} \tau_n(H)}{\max_{G\in \cF_n} \beta_n(G)} \eta_n(G) = \sum_{G\in \cF_n} a_{G,n}\eta_n(G).
\end{align*}
Since $\lambda_n(A_n)\to 0$, the right-hand side of the above inequality does as well. 
Thus, by \Cref{lem:GeneralisationContiguity} and the contiguity of $\eta_n$ and $\tilde\eta_n$, it follows that $\sum_{G\in \cF_n} a_{G,n}\tilde\eta_n(G)\to 0$.
However, by assumption on $\beta_n$ and the inequality $\min_{G\in \cF_n} \beta_n(G)\le \tilde\Psi_n$,
\begin{align*}
\tilde\lambda_n(A_n) = \frac{\sum_{S\in A_n} \tilde\psi_n(S)}{\tilde\Psi_n} \le \sum_{G\in \cF_n} \frac{\sum_{H\in \cF_n:\, G\cup H\in A_n,\, G\cap H = \emptyset} \tau_n(H)}{\min_{G\in \cF_n} \beta(G)} \tilde\eta_n(G) \le \gamma\sum_{G\in \cF_n} a_{G,n}\tilde\eta_n(G),
\end{align*}
implying that $\tilde\lambda_n(A_n)\to 0$.
Since the implication $\tilde\lambda_n(A_n)\to 0\Rightarrow \lambda_n(A_n)\to 0$ is derived similarly, the proof of part (b) follows.
\end{proof}

A combination of \Cref{thm:McKay} and \Cref{prop:general} implies the following useful corollary.
Recall that $\mu_d = \mu_{d,n}$ stands for the uniform distribution on the space $\cG_d(n)$.

\begin{corollary}\label{cor:cont_Gnd}
Consider integers $k = k(n)\ge 2$ and $d_1 = d_1(n),\ldots,d_k = d_k(n)\ge 1$ of sum $D=D(n)$ such that $D^3 \leq n$. 
Fix two sequences of measures $(\eta_n)_{n\ge 1}$ and $(\tilde\eta_n)_{n\ge 1}$ where $\eta_n$ and $\tilde\eta_n$ are defined on the space $\cG_{d_1}(n)$, and set $\tau_n \coloneqq \mu_{d_2,n}\oplus\dotsb\oplus\mu_{d_k,n}$.
Then, each of the following holds.
\begin{enumerate}
    \item[\emph{(a)}] There exists an absolute constant $C>0$ such that, for every $n\ge 1$,
    \[\dtv(\eta_n \oplus \tau_n, \tilde\eta_n \oplus \tau_n) \leq C \cdot \dtv(\eta_n, \tilde\eta_n).\]
    \item[\emph{(b)}] Suppose $(\eta_n)_{n\ge 1}$ and $(\tilde\eta_n)_{n\ge 1}$ are contiguous. Then, $(\eta_n \oplus \tau_n)_{n\ge 1}$ and $(\tilde\eta_n \oplus \tau_n)_{n\ge 1}$ are also contiguous.
\end{enumerate}
\end{corollary}
\begin{proof}
For every $n\ge 1$, define $\lambda_n = \eta_n \oplus \tau_n$ and $\tilde\lambda_n= \tilde\eta_n \oplus \tau_n$.
Also, for every $G\in \cG_{d_1}$, set $\beta_n(G) = \sum_{H\in \cG_{D-d_1}: G\cap H = \emptyset} \tau_n(H)$.
We will show that, for all graphs $G, G'\in \cG_{d_1}$,
\begin{align}
\label{eq:BetaSameMass}
\beta_n(G)=\left(1+O\left(\frac{D^3}{n}\right)\right)\beta_n(G'),
\end{align}
where the implicit constant in the $O(\cdot)$ term does not depend on the parameters of the problem.
Note that both parts of the statement then follow by combining~\eqref{eq:BetaSameMass} and \Cref{prop:general}.

To this end, denote by $\Lambda(G)$ the family of $(k-1)$-tuples $(H_2,\ldots,H_k)\in \cG_{d_2}\times\dotsb\times\cG_{d_k}$ such that $G,H_2,\ldots,H_k$ are pairwise edge-disjoint. 
Then, by writing $D_i = d_1+\dotsb+d_i$ for all $i\in [k]$ and applying Theorem~\ref{thm:McKay} consecutively $k-1$ times for $X_{i-1} \coloneqq G\cup H_2\cup\dotsb\cup H_{i-1}$ and $(g_1,\ldots,g_n) = (d_i,\ldots,d_i)$ for all $i\in [2,k]$, we obtain that 
\[\abs{\Lambda(G)}=\prod_{i=2}^k \frac{(d_i n)!}{(d_i n/2)! \, 2^{d_i n/2} \prod_{j=1}^n d_j!} \exp\left(- \frac{d_i-1}{2} - \bigg(\frac{d_i-1}{2}\bigg)^2 - \frac{d_iD_{i-1}}{2} - O\bigg(\frac{d_i D_i^2}{n}\bigg)\right),\]
where we used that
\[\frac{1}{4e(G_i)}\sum_{j=1}^n d_i(d_i-1) = \frac{d_i-1}{2}\qquad\text{and}\qquad\frac{1}{2e(G_i)}\sum_{j\ell\in E(X_{i-1})} d_i^2 = \frac{d_i^2 |E(X_{i-1})|}{d_i n} = \frac{d_i D_{i-1}}{2}.\]
In particular, since $\sum_{i=2}^k d_i D_i^2\le D^3$ and
\[\beta_n(G) = \frac{|\Lambda(G)|}{|\{(H_2,\ldots,H_k)\in \cG_{d_2}\times\dotsb\times\cG_{d_k} \text{ pairwise edge-disjoint}\}|},\]
we have that \eqref{eq:BetaSameMass} holds, as desired.
\end{proof}

The last lemma in this section shows that, for $m\ge 2$ unlabelled regular graphs on $n$ vertices and with uniformly bounded degrees, by attributing labels in $[n]$ independently to each of them (thus realising them on the same vertex set), the probability that these graphs remain are pairwise edge-disjoint only depends on $m$ and the degrees of the given graphs.
This is shown via a moment computation indicating convergence of the number of repeated edges to a Poisson random variable.
A similar argument is presented in a slightly different context in Section~2.3 in~\cite{Wor99}; a proof is provided for completeness nonetheless.
Given a labelled graph $G$, the \emph{skeleton} of $G$ is the unlabelled graph obtained from $G$ by erasing the labels of its vertices.

\begin{lemma}\label{lem:disjoint}
Fix $m\ge 2$, $d_1,\ldots,d_m\ge 1$ and $m$ unlabelled regular graphs on $n$ vertices with degrees equal to $d_1,\ldots,d_m$, respectively.
For every $i\in [m]$, denote by $H_i$ the (random) graph obtained from the $i$-th unlabelled graph by attributing the labels in $[n]$ bijectively, uniformly at random and independently for different $i$.
Then, the probability that $H_1,\ldots,H_m$ are pairwise edge-disjoint is $\exp(-D/2)+o(1)$ where $D = \sum_{1\le k < \ell \le m} d_k d_\ell$ and the $o(1)$ is uniform over all choices of $H_1,\ldots,H_m$.
\end{lemma}
\begin{proof}
For convenience, we identify each graph with its edge set.
Denote by $X$ the number of repeated edges, that is, edges appearing in at least two of the graphs $H_1,\ldots,H_m$. 
To establish the lemma, it suffices to show that the factorial moments of $X$ converge to the factorial moments of a Poisson random variable with parameter $D/2$ (see, for instance, the case $k=1$ of Lemma~2.8 in \cite{Wor99}).

For every $r\ge 1$, denote by $\Lambda_r$ the set of (ordered) $r$-tuples of distinct edges in $K_n$, and denote by $\Lambda_r'\subseteq \Lambda_r$ the subset of $r$-tuples of edges which do not induce a matching. 
Observe that the $r$-th factorial moment of $X$ is given by the sum $\sum_{(e_1,\ldots,e_r)\in \Lambda_r} \mathbb P(e_1,\ldots,e_r\text{ are repeated edges})$. 
For~distinct edges $e_1,\dots,e_r$, we define the graphs $E_i = H_i\cap \{e_1,\dots, e_r\}$ for all $i\in [m]$.
We first show that the contribution of the $r$-tuples in $\Lambda_r'$ to the $r$-th factorial moment is of order $o(1)$.
\begin{claim}\label{cla:repeated-edges}
For every $r\ge 1$, $\sum_{(e_1,\ldots,e_r)\in \Lambda_r'} \mathbb P(e_1,\ldots,e_r\text{ are repeated edges}) = o(1)$.
\end{claim}\begin{proof}
     Note that 
\begin{equation}\label{eq:fact_moments}
\sum_{(e_1,\ldots,e_r)\in \Lambda_r'} \mathbb P(e_1,\ldots,e_r\text{ are repeated edges}) = \sum_{(e_1,\ldots,e_r)\in \Lambda_r'} \mathbb P(\forall i\in [r], \exists \text{ distinct }j,\ell\in [m], e_i\in H_j\cap H_{\ell}).   
\end{equation} 

Fix $(e_1,\ldots,e_r)\in \Lambda_r'$ and suppose that $e_1,\ldots,e_r$ all are repeated edges, and that $e_1\cup \cdots\cup e_r$ induces a graph $E$ on $t\in [2r-1]$ vertices. 
Denote by $\Psi_m$ the family of $m$-tuples of sets $E_1',\ldots,E_m'\subseteq \{e_1,\ldots,e_r\}$ where every edge among $e_1,\ldots,e_r$ is contained in at least two sets.
In particular, the event that $e_1,\ldots,e_r$ are repeated edges is the same as the event that $E_1,\dots,E_m$ is an $m$-tuple~in~$\Psi_m$.

Now, fix an $m$-tuple $(E_1',\ldots,E_m')$ in $\Psi_m$. For every $i\in [m]$, we assume that $E_i'$ has $s(i)\ge 1$ connected components containing at least one edge, and consider a forest $F_i$ consisting of disjoint trees $T_1^i,\ldots,T_{s(i)}^i$ spanning these components.
Observe that, since $(e_1,\ldots,e_r)\in \Lambda_r'$, we may (and do) select the said forests so that there is a vertex $v\in V(E)$ incident to at least three edges in $\bigcup_{i=1}^m F_i$, counting with multiplicity.

Given $e_1,\ldots,e_r$ and $F_1,\ldots,F_m$, we show that 
\begin{equation}\label{eq:labels}
\mathbb P(\forall i\in [m], F_i\subseteq H_i) \le \prod_{i=1}^m  \bigg(\frac{d_i}{n-2r}\bigg)^{e(F_i)}.
\end{equation}
To this end, we attribute labels to the skeletons of the graphs $H_1,\ldots,H_m$ in this order. 
For every $i\in [m]$, root each of the trees $T_1^i,\ldots,T_{s(i)}^i$ in arbitrary vertices and reveal the labels of the roots in the skeleton~of~$H_i$.
Starting from the roots, we order the remaining vertices in the forest $F_i$ so that every vertex $v$ is preceded by its parent $p(v)$.
Then, for every new vertex $v$, the probability that $v$ is adjacent to $p(v)$ in $H_i$ given the labels of all preceding vertices is at most $d_i/(n-2r)$ since $H_i$ is a $d_i$-regular graph which has received at most $t \le 2r-1$ of its vertex labels. 
Since the forest $F_i$ contains $e(F_i)$ edges,~\eqref{eq:labels} follows.

We show that 
$\sum_{i=1}^m e(F_i) \ge t+1$.
On the one hand, to see that the latter sum is bounded from below by $t$, note that every vertex incident to $e_1\cup\dotsb\cup e_r$ belongs to at least two different forests among $F_1,\ldots,F_m$ and, therefore, is incident to at least two edges in $\bigcup_{i=1}^m F_i$, counting with multiplicity. 
Moreover, to have $\sum_{i=1}^m e(F_i) = t$, every vertex in $E$ must be incident to \emph{exactly} two edges in $\bigcup_{i=1}^m F_i$, counting with multiplicity.
However, this assumption does not hold for the vertex $v$, showing that $\sum_{i=1}^m e(F_i) \ge t+1$.

Thus, for every $r$-tuple of edges in $\Lambda_r'$ and every choice of forests $F_1,\ldots,F_m$,~\eqref{eq:labels} is of order $O(n^{-t-1})$.
Since there are $O(n^t\cdot t^{2r})$ choices for the edges $e_1,\ldots,e_r$ and $O(2^{rm})$ choices for the forests $F_1,\ldots,F_m$, we conclude~\eqref{eq:fact_moments} is $o(1)$, as desired.
\end{proof}

By \Cref{cla:repeated-edges}, it suffices to estimate
\begin{equation}\label{eqn:goal}
\Sigma \coloneqq \sum_{(e_1,\dots, e_r)\in \Lambda_r\setminus \Lambda_r'} \mathbb P(e_1,\ldots,e_r\text{ are repeated edges})\,.
\end{equation}
To begin with, fix $(e_1,\dots, e_r)\in \Lambda_r\setminus \Lambda_r'$ and observe that
\begin{equation}\label{eq:Lambda}
|\Lambda_r\setminus \Lambda_r'| = \binom{n}{2r} \frac{(2r)!}{2^r} = (1+o(1)) \frac{n^{2r}}{2^r}.
\end{equation}
By utilising the idea of label attribution from the proof of \Cref{cla:repeated-edges} again, we conclude that the contribution coming from realisations of $E_1,\ldots,E_m$ where some of $e_1,\ldots,e_r$ is repeated at least three times is $o(1)$.

We end the proof by analysing the following procedure. We process the edges $e_1,\ldots,e_r$ consecutively. 
For every $i\in [r]$, we expose the labels of the endpoints of the edge $e_i =u_iv_i$ in each of the (unlabelled) graphs $H_1,\ldots,H_m$ and check if these graphs contain the edge $e_i$ or not.

Note that the events $e_i\in H_\ell$ and $e_i\in H_k$ are independent for distinct $\ell, k\in [m]$.
Next, we analyse the probability that $e_i$ is contained in $H_j$ for some fixed $j\in [m]$ conditionally on the vertex labels $u_1,\dots, u_{i-1},v_1,\dots v_{i-1}$ being exposed in $H_j$.
Note that $H_j\setminus \{u_1,v_1,\ldots,u_{i-1},v_{i-1}\}$ contains at least $nd_j/2 - 2(i-1)d_j = (1+o(1))nd_j/2$ edges, and the probability that $e_i$ is one of them is
\[(1+o(1))\frac{nd_j}{2}\bigg/\binom{n-2(i-1)}{2} = (1+o(1))\frac{d_j}{n}\,.\]
Hence, for every $i\in [r]$, conditionally on the past steps, the edge $e_i$ belongs to precisely two graphs among $H_1,\ldots,H_m$ with probability 
\begin{equation}\label{eq:1term}
(1+o(1)) \sum_{1\le k < \ell \le m} \frac{d_k d_\ell}{n^2} = (1+o(1)) \frac{D}{n^2}.
\end{equation}
By multiplying the latter probabilities for all $i\in [r]$ and using~\eqref{eq:Lambda}, it follows that
\begin{align*}
    \sum_{(e_1,\dots, e_r)\in \Lambda_r\setminus \Lambda_r'} \mathbb P(e_1,\ldots,e_r\text{ are repeated edges}) = |\Lambda_r\setminus \Lambda_r'| \left((1+o(1)) \frac{D}{n^2}\right)^r =  (1+o(1)) \left(\frac{D}{2}\right)^r.
\end{align*}
Combining this with \Cref{cla:repeated-edges}, we deduce that the $r$-th factorial moment of $X$ is equal to $(1+o(1)) (D/2)^r$, as desired.
\end{proof}

\section{Main technical result: decomposing one matching at a time}\label{sec:technical}

This section is dedicated to the proof of the following technical proposition.
For the sake of coherence, we state the following results with the same restriction of $d\le n^{1/10}$. The proof works for a slightly larger range of $d$, but clearly cannot go pass $n^{1/7}$. We made no attempt at determining the best possible range.
\begin{proposition}\label{prop:PM-counting}
There exists a constant $C_0> 0$ such that, for all even $n$ and $d = d(n)\in [3,n^{1/10}]$, the number of perfect matchings $Y$ in $G(n,d)$ satisfies
\[\Prob(|Y-\mathbb{E}[Y]| \geq d^{-1.1}\mathbb{E}[Y])\leq C_0d^{-1.1}.\]
\end{proposition}

The following consequence of \Cref{thm:strassen} and \Cref{prop:PM-counting} will be used in the proofs of each of our main results.

\begin{corollary}\label{cor:dTV}
There exists a constant $C_1 > 0$ such that, for all even $n$ and $d = d(n)\in [3,n^{1/10}]$,
\[\dtv(\mu_{d} \oplus \mu_1, \mu_{d+1}) \leq C_1 d^{-1.1}.\]
\end{corollary}
\begin{proof}[Proof of \Cref{cor:dTV} assuming \Cref{prop:PM-counting}]
Consider an auxiliary bipartite graph $H$ with parts 
\[S = \{(G_d,G_{d+1})\in \cG_d(n)\times \cG_{d+1}(n): G_d\subseteq G_{d+1}\}\] 
and $T = \cG_{d+1}(n)$ where $(G_d,G_{d+1})\in S$ is adjacent to $G_{d+1}'\in T$ if $G_{d+1}=G_{d+1}'$.
Note that every vertex in $S$ has exactly one neighbour in $T$ and thus $|E(H)|=|S|$.
Moreover, the degree of a vertex $G$ in $T$ is equal to the number of perfect matchings in $G$, meaning that $|E(H)|=|T|\cdot \mathbb{E}[Y]$, where $Y$ is the number of perfect matchings in $G(n,d+1)$.
By \Cref{prop:PM-counting}, at least $(1-C_0 d^{-1.1})|T|$ of the vertices in $T$ have at least $(1-d^{-1.1})\mathbb{E}[Y]=(1-d^{-1.1})|E(H)|/|T|$ neighbours in $S$.
By applying \Cref{thm:strassen} to the graph $H$ with $\eps=d^{-1.1}$ and $\delta=C_0d^{-1.1}$, we find a coupling $(X,Z)$ with $X$ uniformly distributed on $S$ and $Z$ uniformly distributed on $T$ such that 
\begin{align*}
    \mathbb{P}(XZ \notin E(H)) = O(d^{-1.1}).
\end{align*}
However, by definition, the second marginal $X'$ of $X$ is distributed according to $\mu_d\oplus\mu_1$, and thus
\begin{align*}
    \dtv(\mu_{d} \oplus \mu_1, \mu_{d+1}) \leq \mathbb{P}(X' \neq Z) = \mathbb{P}(XZ \notin E(H)) = O(d^{-1.1}),
\end{align*}
which concludes the proof.
\end{proof}

In the remainder of this section, we focus on the proof of \Cref{prop:PM-counting}.
Recall that $Y$ stands for the number of perfect matchings in the random $d$-regular graph $G(n,d)$. 
To show \Cref{prop:PM-counting}, one may be tempted to compute the first two moments of $Y$ and apply Chebyshev's inequality.
While such an estimate exists in the literature~\cite{Gao23}, it only shows that $\mathbb P(|Y-\mathbb E[Y]|\ge d^{-1}\mathbb E[Y]) = O(d^{-1})$, which is insufficient for our purposes due to the divergence of the harmonic sum.
We extend this approach by analysing higher moments.

Computing higher moments of $Y$ directly becomes a complicated task.
One elegant approach for studying the number of perfect matchings relies on a method known as orthogonal decomposition and projection introduced by Janson~\cite{Jan94} and first applied in the setting of random regular graphs by Gao~\cite{Gao23}.
Here, we follow Gao's presentation in~\cite{Gao23}.
We define $X$ to be the number of triangles in $G(n,d)$ and $Y^* = aX + b$ with $a=\mathrm{Cov}(X,Y)/ \mathrm{Var}[X]$ and $b = \mathbb{E}[Y] - a\mathbb{E}[X]$, where we recall that $\mathrm{Cov}(X,Y) = \mathbb E[XY]-\mathbb E[X]\mathbb E[Y]$.
Then, by choice of $a$ and $b$, it is easy to verify that $\mathbb{E}[Y^*] = \mathbb{E}[Y]$ but it also holds that $Y^*$ minimises $\mathbb{E}[(Y-Y^*)^2]$. 
We will control the moments of $Y - Y^*$ and $Y^*$ separately. 

First, we collect several relevant estimates due to Gao~\cite{Gao23,Gao23b}. 
We recall that asymptotic notation is used with respect to $n\to\infty$ and hidden constants are universal.
\begin{claim}[\cite{Gao23,Gao23b}]\label{cla:estimate-1}
Consider $d = d(n)\in [3,n^{1/10}]$. Each of the following holds.
\begin{enumerate}
\setlength\itemsep{0.3em}
\item[\emph{1.}] $\mathrm{Cov}(X,Y) = O(\mathbb{E}[X]\mathbb{E}[Y]/d^3)$,
\item[\emph{2.}] $\mathbb{E}[X] = (1+O(n^{-1}))(d-1)^3/6$.
\item[\emph{3.}] $\mathrm{Var}(X) = (1+O(n^{-1}))(d-1)^3/6$.
\item[\emph{4.}] $\mathbb{E}[Y^2] = (1 + d^{-3}/6 + O(d^{-4} + d^{-3}/n + \sqrt{d/n}(\log n)^3)) \mathbb{E}[Y]^2$.
\end{enumerate}
\end{claim}
The first and fourth points are from~\cite[Theorem~10]{Gao23}, the second point is from~\cite[Theorem~9]{Gao23}, and the third point follows from~\cite[Section 4.2]{Gao23b}. 
Another tool used in our subsequent analysis is the following theorem from~\cite{Gao23b} estimating the probability that an edge belongs to a random regular graph conditionally on some part of it being already exposed.

\begin{theorem}[Theorem 2 in~\cite{Gao23b}]\label{thm:subgraph-prob} 
Fix $d=o(n)$, a random $d$-regular graph $G\sim G(n,d)$ and a graph $H$ on the same vertex set and with maximum degree at most $d$. 
Suppose that $dn - |E(H)| = \Omega(dn)$ and that $uv\notin E(H)$ for some vertices $u,v$. 
Then, 
    \[
        \Prob(uv\in G\mid H\subseteq G) =  \frac{(d-\deg_H(u))(d-\deg_H(v))}{dn}\left(1 - \frac{\phi_H(uv)}{dn}\right)\left(1 + O\left(\frac{|E(H)|}{n^2} + \frac{|E(H)|^2}{d^2n^2} + \frac{d^2}{n^2}\right)\right)
    \]
    where 
    \[\phi_H(uv) = \deg_H(u) \deg_H(v) + \!\! \sum_{x\in N_H(u)} \!\! \deg_H(x) + \!\! \sum_{y\in N_H(v)} \!\! \deg_H(y) - d - 2|E(H)| - (d-1)(\deg_H(u) + \deg_H(v))\]
    with $N_H(u)$ and $N_H(v)$ denoting the neighbourhoods of $u$ and $v$ in $H$, respectively.
    Thus, if we have $|E(H)|=O(1)$ and $d^2\le n$, then
     \[
        \Prob(uv\in G\mid H\subseteq G) = \left(1 +O\left( \frac{1}{n} \right)\right)\frac{(d-\deg_H(u))(d-\deg_H(v))}{dn}\,.
    \]
\end{theorem} 

In~\cite[Section~5.2]{Gao23b}, it is shown that $\mathbb{E}[X^k] = (1+o(1))\mathbb{E}[X]^k$ for all fixed $k\ge 1$ when $d=\omega(1)$ and~$d = o(n^{1/2})$.
However, this estimate is not sufficiently precise for our purposes when $k \in \{2,3,4\}$.
The goal of the next lemma is to go further in the understanding of the error term when $d\le n^{1/10}$.
We remark that, while the same arguments could be extended to larger values of $k$ at the cost of an additional technical effort (and smaller values for $d$), we focus on the cases $k \in \{2,3,4 \}$ relevant to our proof. 

\begin{lemma}\label{lem:estimate-2}
Consider $d = d(n)\in [3,n^{1/10}]$. 
Then, for each $k\in \{2,3,4\}$,
\[\mathbb{E}[X^k] = \mathbb{E}[X]^k + \binom{k}{2}\mathbb{E}[X]^{k-1} + O(\mathbb{E}[X]^{k-2})\,.\]
\end{lemma}
\begin{proof}
For $k=2$, the proof follows directly from~\Cref{cla:estimate-1}.
Thus, we are left to handle the cases $k=3$ and $k=4$.
Set $N = \binom{n}{3}$ and define $H_1,\dots,H_N$ to be an arbitrary ordering of all triangles in $K_n$. 
Also, define $X_i$ to be the indicator random variable of the event $H_i\subseteq G(n,d)$, so we have~$X = \sum_{i=1}^N X_i$.

\vspace{1em}
\paragraph*{The case $k=3$.}
We will in fact show that $\mathbb{E}[X_{(3)}]=\mathbb{E}[X]^3+O(\mathbb{E}[X])$, where $X_{(3)}=X(X-1)(X-2)$ is the third factorial moment of $X$. As $\mathbb{E}[X_{(3)}]=\mathbb{E}[X^3]-3\mathbb{E}[X^2]+2\mathbb{E}[X]$, the result will follow by applying the lemma for $k=2$.
We now compute $\mathbb{E}[X_{(3)}] = \sum_{(i,j,k)\in \binom{[N]}{3}}\mathbb{E}[X_iX_jX_k]$ by distinguishing cases. 
Note that, by a simple application of the second part of \Cref{thm:subgraph-prob},
for all distinct $i,j,k\in [n]$, $R = H_i \cup H_j \cup H_k$, any subgraph $S$ of $R$ and for every edge $uv \in E(R) \setminus E(S)$,
\begin{align}
\label{eq:Proba-Single-Edge-Added}
 \Prob(uv\in G(n,d)\mid S\subseteq G(n,d)) =  \frac{(d-\deg_S(u))(d-\deg_S(v))}{dn}\left(1 +O\left( \frac{1}{n} \right)\right).
\end{align}

We write $\mathbb{E}[X_{(3)}] =\Sigma_1 + \Sigma_2$ where $\Sigma_1$ is the contribution to $\mathbb E[X_{(3)}]$ by the triplets $(H_i, H_j, H_k)$ whose union contains more edges than vertices, and $\Sigma_2$ is the contribution to $\mathbb E[X_{(3)}]$ by the triplets $(H_i, H_j, H_k)$ whose union contains at most as many edges as vertices.
We start with the following claim.
\begin{Claim}
\label{Cl:Sigma1}
$\Sigma_1 = O(\mathbb E[X])$.
\end{Claim}
\begin{proof}
Fix a triplet $(H_i, H_j, H_k)$ whose union $R=H_i\cup H_j\cup H_k$ contains $h = |E(R)|\in [5,9]$ edges and at most $h-1$ vertices.
Consider an arbitrary ordering $e_1, \dots, e_h$ of the edges of $R$, and denote by $R_i$ the graph induced by $e_1, \dots, e_i$.
By \eqref{eq:Proba-Single-Edge-Added}, for every $i\in [h]$,
    \begin{align}
    \label{eq:Proba-Single-Edge-Added-Case1}
        \Prob(e_i \subseteq G(n,d)\mid R_{i-1}\subseteq G(n,d)) =  \left(1 +O\left( \frac{1}{n} \right)\right)\frac{(d-\deg_{R_{i-1}}(u))(d-\deg_{R_{i-1}}(v))}{dn} = O \left(\frac{d}{n} \right).
    \end{align}
    Therefore, by multiplying the probabilities in~\eqref{eq:Proba-Single-Edge-Added-Case1} for every $i\in [h]$, we obtain that
    \begin{align}
    \label{eq:Proba-Subgraph-Case1}
        \Prob(R \subseteq G(n,d)) = O\bigg(\bigg(\frac{d}{n}\bigg)^{h}\bigg).
    \end{align}
By combining~\eqref{eq:Proba-Subgraph-Case1}, a simple union bound over the values $h\in [5,9]$, the inequality $d^{10}\le n$ and the second point of \Cref{cla:estimate-1}, we obtain that
    \begin{align*}
        \Sigma_1 = \sum_{h=5}^9  O\bigg(n^{h-1} \bigg(\frac{d}{n}\bigg)^{h}\bigg) = O\bigg(\frac{d^{9}}{n}\bigg) = O(d^3) = O(\mathbb E[X]),
    \end{align*}
    as desired.
\end{proof}

We now estimate $\Sigma_2$, that is, the contribution of the triplets $(H_i, H_j, H_k)$ such that their union $R$ satisfies $|E(R)|\le |V(R)|$.
In this case, observe that actually $|E(R)| = |V(R)|$ since every connected component of $R$ contains a cycle, and therefore has at least as many edges as vertices.
Moreover, every connected component of $R$ contains a single cycle, implying that different triangles in the triplet $(H_i, H_j, H_k)$ must be (vertex-)disjoint.

Consider an arbitrary ordering $e_1, \dots, e_9$ of the edges of $R$, and denote by $R_i$ the graph induced by $e_1, \dots, e_i$ for all $i\in [9]$. By \eqref{eq:Proba-Single-Edge-Added}, we have
\begin{align*}
    \Prob(e_i \subseteq G(n,d)\mid R_{i-1}\subseteq G(n,d)) =  \left(1 +O\left( \frac{1}{n} \right)\right)\frac{(d-\deg_{R_{i-1}}(u))(d-\deg_{R_{i-1}}(v))}{dn}.
\end{align*}
Without loss of generality, assume each of $\{e_1,e_2,e_3\}$, $\{e_4,e_5,e_6\}$ and $\{e_7,e_8,e_9\}$ form a triangle. 
By multiplying the said probabilities, we obtain
\begin{align*}
    \Prob(R \subseteq G(n,d)) = \left(1+O\left(\frac{1}{n}\right)\right)^9\bigg(\frac{d^2}{dn}\cdot \frac{d(d-1)}{dn}\cdot \frac{(d-1)^2}{dn}\bigg)^3 = \left(1+O\left(\frac{1}{n}\right)\right) \frac{(d-1)^9}{n^9}.
\end{align*}
By summing over all triplets $(H_i,H_j,H_k)$ giving rise to $R$ as described above and using that $d^{10}\le n$, it follows that
{\begin{align}
\label{eq:Sigma2c}
\hspace{-0.7em}\Sigma_{2} = \binom{n}{3}\binom{n-3}{3}\binom{n-6}{3}\cdot \left(1+O\left(\frac{1}{n}\right)\right)\frac{(d-1)^9}{n^9}  = \left(1+O\left(\frac{1}{n}\right)\right)\frac{(d-1)^9}{216} = \mathbb E[X]^3 + O(\mathbb E[X]).
\end{align}
Thus, by combining \Cref{Cl:Sigma1} and \eqref{eq:Sigma2c}, we obtain that $\mathbb{E}[X_{(3)}]=\mathbb E[X]^3 + O(\mathbb E[X])$, as desired.

\vspace{1em}
\paragraph*{The case $k=4$.}
Analogously to the previous case, we show that $\mathbb{E}[X_{(4)}]=\mathbb{E}[X]^4+O(\mathbb{E}[X]^2)$, where $X_{(4)}=X(X-1)(X-2)(X-3)$. As $\mathbb{E}[X_{(4)}]=\mathbb{E}[X^4]-6\mathbb{E}[X^3]+11\mathbb{E}[X^2]-6\mathbb{E}[X]$, the result will follow from the cases $k=2$ and $k=3$ of the current lemma.
Having $\mathbb{E}[X_{(4)}] = \sum_{(i,j,k,\ell)\in \binom{[N]}{4}}\mathbb{E}[X_iX_jX_kX_{\ell}]$, we again write $\mathbb{E}[X_{(4)}] =\Sigma_1 + \Sigma_2$, where $\Sigma_1$ is the contribution to $\mathbb E[X_{(4)}]$ of the tuples $(H_i, H_j, H_k, H_{\ell})$ whose union contains more edges than vertices, and $\Sigma_2$ is the contribution to $\mathbb E[X_{(4)}]$ of the tuples $(H_i, H_j, H_k, H_{\ell})$ whose union contains at most as many edges as vertices. 
We start with a claim analogous to \Cref{Cl:Sigma1} for the case $k=4$.
\begin{Claim}
\label{Cl:Sigma1-k=4}
$\Sigma_1 = O(\mathbb E[X]^2)$.
\end{Claim}
\begin{proof}
Similarly to the proof of \Cref{Cl:Sigma1}, fix a tuple $(H_i, H_j, H_k, H_\ell)$ whose union $R=H_i\cup H_j\cup H_k\cup H_\ell$ contains $h = |E(R)|\in [5,12]$ edges and at most $h-1$ vertices.
Similarly to \eqref{eq:Proba-Subgraph-Case1}, we obtain
    \begin{align}
    \label{eq:Proba-Subgraph-Case-k=4}
        \Prob(R \subseteq G(n,d)) = O\bigg(\bigg(\frac{d}{n}\bigg)^h\bigg).
    \end{align}
By combining~\eqref{eq:Proba-Subgraph-Case-k=4}, a simple union bound over the values $h\in [5,12]$, the inequality $d^{10}\le n$ and the second point of \Cref{cla:estimate-1}, we obtain that
    \begin{align*}
        \Sigma_1 =\sum_{h=5}^{12}  O\bigg(n^{h-1} \bigg(\frac{d}{n}\bigg)^h\bigg) = O\left(\frac{d^{12}}{n}\right) = O(d^6) = O(\mathbb E[X]^2),
    \end{align*}
    as desired.
\end{proof}
We now estimate $\Sigma_2$.
As in the case $k=3$, all tuples $(H_i, H_j, H_k, H_{\ell})$ with union $R$ contributing to $\Sigma_2$ satisfy that every connected component of $R$ contains a single cycle and, therefore, any two triangles must be (vertex-)disjoint. 
Similarly to \eqref{eq:Sigma2c}, we obtain
\begin{align}
\label{eq:Sigma2-k=4}
    \Sigma_{2} &= \binom{n}{3}\binom{n-3}{3}\binom{n-6}{3}\binom{n-9}{3}\cdot \left(1+O\left(\frac{1}{n}\right)\right)\frac{(d-1)^{12}}{n^{12}} = \mathbb E[X]^4+O(\mathbb E[X]^2).
\end{align}
Thus by \Cref{Cl:Sigma1-k=4} and \eqref{eq:Sigma2-k=4} we obtain that $\mathbb{E}[X_{(4)}] = \mathbb{E}[X]^4 + O(\mathbb{E}[X]^2)$, completing the proof of \Cref{lem:estimate-2}.}
\end{proof}

Finally, we estimate $\mathbb{E}[(Y-Y^*)^2]$.
Our proof closely follows the proof of Claim 12 in Gao's work~\cite{Gao23}, but we use \Cref{lem:estimate-2} to obtain a more precise estimate.

\begin{proposition}
\label{prop:Estimate-E[(Y-Y*)^2]}
Consider $d = d(n)\in [3,n^{1/10}]$. Then, $\mathrm{Var}[Y-Y^*] = O(\mathbb{E}[Y]^2/d^4)$.
\end{proposition}
\begin{proof}
By definition of $Y^* = aX+b$,
\[\mathrm{Cov}(Y^{*},Y^{*}-Y) = \mathrm{Var}[Y^*] -\mathrm{Cov}(Y^{*},Y) = a^2\mathrm{Var}[X] - a\mathrm{Cov}(X,Y) =0.\] 
Moreover, by the above equality and the fact that $\mathbb{E}[Y]=\mathbb{E}[Y^{*}]$, we obtain that 
\begin{equation}\label{eq:Ys}
\begin{split}
\mathrm{Var}[Y-Y^*] = \mathrm{Cov}(Y,Y-Y^*) 
&= (\mathbb E[Y^2] - \mathbb E[(Y^*)^2]) + (\mathbb E[(Y^*)^2] - \mathbb E[YY^*])\\
&=(\mathrm{Var}[Y] - \mathrm{Var}[Y^*]) + \mathrm{Cov}(Y^*,Y^*-Y) \\
&= \mathrm{Var}[Y]-\mathrm{Var}[Y^*].
\end{split}
\end{equation}
On the one hand, by the fourth point in~\Cref{cla:estimate-1}
and the fact that $d^{10}\le n$, we deduce that
\begin{align}
\label{eq:VarY}
    \mathrm{Var}[Y] = \left( \frac{1}{6d^{3}} + O \left(\frac{1}{d^{4}} + \frac{d^3}{n} + \sqrt{\frac{d}{n}} (\log n)^3 \right)  \right)\mathbb{E}[Y]^2 = \frac{\mathbb{E}[Y]^2}{6d^{3}} + O\left(\frac{\mathbb{E}[Y]^2}{d^{4}}\right).
\end{align}
On the other hand, by the first point in~\Cref{cla:estimate-1}
and \Cref{lem:estimate-2} for $k=2$, we have 
\begin{align*}
     \mathrm{Var}[Y^*] = \frac{\mathrm{Cov}(X,Y)^2}{\mathrm{Var}[X]} &= \frac{(d^{-3}+O(d^{-4}+d/n))^2(\mathbb{E}[X]\mathbb{E}[Y])^2}{\mathbb{E}[X]+O(1)}.
\end{align*}
Lastly, using the second point in \Cref{cla:estimate-1} and the fact that $d^{10}\le n$, we obtain that
\begin{align}
\label{eq:VarY*}
    \mathrm{Var}[Y^*]= \frac{\mathbb{E}[Y]^2}{6d^3} + O\left(\frac{\mathbb{E}[Y]^2}{d^4}\right).
\end{align}
Therefore, by combining \eqref{eq:Ys}, \eqref{eq:VarY} and \eqref{eq:VarY*}, it follows that
\begin{align*}
\mathrm{Var}[Y-Y^*] = \mathrm{Var}[Y]-\mathrm{Var}[Y^*] =O(\mathbb{E}[Y]^2/d^4),
\end{align*}
as desired.
\end{proof}

We are ready to finish the proof of \Cref{prop:PM-counting}.

\begin{proof}[Proof of \Cref{prop:PM-counting}]
First, note that
\begin{equation}\label{eqn:decompose}
\mathbb P(|Y-\mathbb E[Y]|\ge d^{-1.1} \mathbb E[Y])\le \mathbb P(|Y-Y^*|\ge d^{-1.1} \mathbb E[Y]/2) + \mathbb P(|Y^*-\mathbb E[Y^*]|\ge d^{-1.1} \mathbb E[Y^*]/2).
\end{equation}
For the first term, Chebyshev's inequality and \Cref{prop:Estimate-E[(Y-Y*)^2]} implies that
\begin{equation}\label{eqn:decompose1}
\mathbb P(|Y-Y^*|\ge d^{-1.1} \mathbb E[Y]/2)\le \frac{\mathrm{Var}[Y-Y^*]}{d^{-2.2} \mathbb E[Y]^2/4} = O(d^{-1.8}).
\end{equation}
To bound the second term, we express $\mathbb E[(Y^*-\mathbb E[Y^*])^4]$ as a function of the moments of $X$. We have that
\begin{align*}
    a^{-4} \mathbb E[(Y^*-\mathbb E[Y^*])^4] 
&= \mathbb E[(X-\mathbb E[X])^4] = \mathbb E[X^4]-4\mathbb E[X^3]\mathbb E[X]+6\mathbb E[X^2]\mathbb E[X]^2-3\mathbb E[X]^4.
\end{align*}
By applying \Cref{lem:estimate-2} for each $k\in \{2,3,4\}$, we obtain that
\begin{align*}
    a^{-4} \mathbb E[(Y^*-\mathbb E[Y^*])^4] = \binom{4}{2}\mathbb E[X]^3 - 4\binom{3}{2}\mathbb E[X]^3 + 6 \binom{2}{2}\mathbb E[X]^3 + O(\mathbb E[X]^2) = O(\mathbb E[X]^2).
\end{align*}
Combining this with the definition of $a$, \Cref{cla:estimate-1}, the fact that $\mathbb E[X] = \Theta(\mathrm{Var}[X])=\Theta(d^3)$ from the $k=2$ case of \Cref{lem:estimate-2}, and Markov's inequality yields
\begin{align*}
\mathbb P(|Y^*-\mathbb E[Y^*]|\ge d^{-1.1}\mathbb E[Y]/2)
&= \mathbb P(|Y^*-\mathbb E[Y^*]|^4\ge d^{-4.4}\mathbb E[Y]^4/16)\le \frac{\mathbb E[(Y^*-\mathbb E[Y^*])^4]}{d^{-4.4}\mathbb E[Y]^4/16}\\
&= O\left(\frac{a^4\mathbb E[X]^2}{d^{-4.4}\mathbb E[Y]^4}\right) = O\left(\frac{\mathbb E[X]^2}{d^{12-4.4}}\right) = O(d^{-1.6}).
\end{align*}
Together with~\eqref{eqn:decompose} and~\eqref{eqn:decompose1}, this finishes the proof.
\end{proof}

\section{\texorpdfstring{Proofs of \Cref{thm:Coupling-Sprinkling-2-Random,thm:Contiguity-Disjoint-Matchings}}{Proofs of Theorem \ref{thm:Coupling-Sprinkling-2-Random} and \ref{thm:Contiguity-Disjoint-Matchings}}}\label{sec:proofs45}

With \Cref{cor:dTV} in hand, the proofs of \Cref{thm:Coupling-Sprinkling-2-Random,thm:Contiguity-Disjoint-Matchings} are relatively easy.
We start with a proof of \Cref{thm:Coupling-Sprinkling-2-Random}.

\begin{proof}[Proof of \Cref{thm:Coupling-Sprinkling-2-Random}.]
We assume without loss of generality that $d_1\le d_2-d_1$ and $d_1'\le d_2-d_1'$; in particular, $d_1+d_1'\le d_2$.
For $r,s$ with $r+s\le d_2$, define $\eta_{r,s} = \mu_r \oplus \mu_s \oplus \mu_1 \oplus \dotsb \oplus \mu_1$ where $\mu_1$ is repeated $d_2-r-s$ times. By the triangle inequality, we have 
\begin{align}
\label{eq:Triangle-ineq-Sprinkling-2-Random}
\dtv(\mu_{d_1} \oplus \mu_{d_2-d_1},\mu_{d'_1} \oplus \mu_{d_2-d_1'}) \leq \dtv(\mu_{d_1} \oplus \mu_{d_2-d_1},\eta_{d_1,d_1'})+ \dtv(\eta_{d_1,d_1'},\mu_{d'_1} \oplus \mu_{d_2-d_1'}).
\end{align}
Since the two terms are treated similarly, we focus on bounding the first one.
Again, by the triangle inequality,
\begin{align}
\label{eq:Triangle-ineq-Sprinkling-2-Random-2}
\dtv(\mu_{d_1} \oplus \mu_{d_2-d_1},\eta_{d_1,d_1'}) = \dtv(\eta_{d_1,d_2-d_1},\eta_{d_1,d_1'}) \leq \sum_{j=d_1'}^{d_2-d_1-1} \dtv(\eta_{d_1,j+1}, \eta_{d_1,j}).
\end{align}
By combining \Cref{cor:cont_Gnd}(a) and \Cref{cor:dTV}, we have $\dtv(\eta_{d_1,j+1}, \eta_{d_1,j}) = O(j^{-1.1})$ hold for every $j \in [d_2]$, with an absolute constant hidden in the $O(\cdot)$ notation. 
By plugging this bound into \eqref{eq:Triangle-ineq-Sprinkling-2-Random-2} and using that $d_1'=\omega(1)$, we obtain that
\begin{align*}
\dtv(\mu_{d_1} \oplus \mu_{d_2-d_1},\eta_{d_1,d_1'})=O\left(\sum_{j=d_1'}^{d_2-d_1-1} j^{-1.1}\right) = O((d_1')^{-0.1}) = o(1).
\end{align*}
The second term in \eqref{eq:Triangle-ineq-Sprinkling-2-Random} is bounded from above in a similar way, which leads to the desired result.
\end{proof}

Next, we prove \Cref{thm:Contiguity-Disjoint-Matchings}.

\begin{proof}[Proof of \Cref{thm:Contiguity-Disjoint-Matchings}.]
For every $i\in [n-1]$, recall the probability distribution $\nu_i = \mu_1\oplus\dotsb\oplus \mu_1$ repeated $i$ times and define $\lambda_i = \mu_i\oplus \nu_{d-i}$.
Fix a sequence $(A_n)_{n\in 2\mathbb N}$ with $A_n\subseteq \cG_d(n)$ for all $n\in 2\mathbb N$ and $\nu_d(A_n)\to 0$.
For every $\eps > 0$, fix $d_1 = d_1(\eps)\ge 2$ with the property that $\sum_{i=d_1}^{\infty} i^{-1.1} < \eps$. 
Again, by combining the triangle inequality, \Cref{cor:cont_Gnd}(a) and \Cref{cor:dTV}, we obtain
\begin{align}
\label{eq:Bound-Diff-Measure-An}
|\lambda_{d_1}(A_n) - \mu_d(A_n)| \leq \dtv(\lambda_{d_1}, \mu_d) \le \sum_{i=d_1}^{d-1} \dtv(\lambda_i, \lambda_{i+1}) = O\bigg(\sum_{i=d_1}^{d-1} \dtv(\mu_i\oplus \mu_1, \mu_{i+1})\bigg) = O(\eps)
\end{align}
with some absolute constant hidden in the $O(\cdot)$ term.
Moreover, by combining \Cref{thm:const} and \Cref{cor:cont_Gnd}(b), we have that $\lambda_0=\nu_d$ are $\lambda_{d_1}$ are contiguous.
Therefore, $\lambda_{d_1}(A_n) \rightarrow 0$.
Plugging this into \eqref{eq:Bound-Diff-Measure-An} implies that $\limsup_{n\to \infty} \mu_d(A_n) = O(\eps)$. 
Since the latter holds for every $\eps > 0$,
this implies that $\mu_d(A_n)\to 0$, as desired.
\end{proof}

\section{\texorpdfstring{Coupling lemmas: prelude to the proof of \Cref{thm:Coupling-Inclusion}}{Coupling lemmas: prelude to the proof of Theorem \ref{thm:Coupling-Inclusion}}}\label{sec:coupling}

In this section, we present several technical lemmas in preparation for the proof of \Cref{thm:Coupling-Inclusion}.

To begin with, we define the \emph{isomorphism class} of a graph $G$ to be the set of graphs $H$ for which $V(G) = V(H)$ for which there exists a bijective map $\sigma: V(G)\to V(H)$ such that $\sigma(G) = H$, that is, for every pair of distinct vertices $u,v\in V(G)$, $uv\in E(G)$ if and only if $\sigma(u)\sigma(v)\in E(H)$.
For fixed $n\ge 1$ and $d\in [n-1]$, we denote by $\cC_d = \cC_d(n)$ the family of isomorphism classes in $\cG_d(n)$; in particular, $\cC_d$ is a partition of the set of $d$-regular graphs on $n$ vertices.
We define the probability distributions $\bar\mu_d = \bar\mu_{d,n}$ and $\bar\nu_d = \bar\nu_{d,n}$ on the space $\cC_d$ by setting,
for each $A\subseteq \mathcal C_d$,
\[\bar\mu_{d,n}(A) = \sum_{S\in A}\sum_{G\in S} \mu_{d,n}(G)\qquad \text{and}\qquad \bar\nu_{d,n}(A) = \sum_{S\in A}\sum_{G\in S} \nu_{d,n}(G).\]
Note that for every class $S\in \cC_d$ and for every graph $G\in S$, we have $\mu_d(G) |S| = \bar\mu_d(S)$ since $\mu_d(G) = 1/|\mathcal G_d(n)|$ for all $G\in \mathcal G_d(n)$. Moreover, the same holds for $\nu_d$, that is, $\nu_d(G) |S| = \bar\nu_d(S)$: indeed, every two graphs in the same isomorphism class have the same number of 1-factorisations. 

Next, we fix an integer $d\ge 1$, define a procedure for sampling a union of $k d$ disjoint perfect matchings on $n$ vertices based on the isomorphism class structure of $\cG_d(n)$, and show that it recovers the distribution~$\nu_{k d,n}$.

\begin{definition}[Alternative Sampling Procedure (ASP)]\label{def:procedure-sample}
Fix a sequence of independent and identically distributed random variables $(Y_i)_{i=1}^{\infty}$ on the space $\cC_d(n)$ with distribution $\bar\nu_{d,n}$.
For every $k\ge 1$, conditionally on a realisation of $Y_1,\ldots,Y_k$, we define the random graph $H_1\oplus \dotsb \oplus H_k$ to be a random element of the set $Y_1\oplus \dotsb\oplus Y_k$ where, for every $i\in [k]$, $H_i$ is sampled according to the uniform distribution on the isomorphism class $Y_i$ conditionally on $H_1,\ldots,H_k$ being edge-disjoint.
We denote the distribution of $H_1\oplus \dotsb \oplus H_k$ by $\eta_{k,d,n}$ or simply $\eta_{k,d}$.
\end{definition}

\begin{lemma}\label{lem:coincide}
For every $k,d\ge 1$ and every graph $G\in \cG_{kd}$, $\eta_{k,d}(G) = (1+o(1))\nu_{kd}(G)$ where the $o(1)$ is uniform over all choices of $G\in \cG_{kd}$.
\end{lemma}
\begin{proof}
Fix $k,d\ge 1$ and recall that, by \Cref{lem:disjoint}, for any choice of $k$ (not necessarily distinct) classes $A_1,\ldots,A_k\in \cC_d$, the proportion of the $k$-tuples of graphs in $A_1\times A_2\times \ldots\times A_k$ such that every two graphs are edge-disjoint is $\e^{-D/2}+o(1)$ where $D = \tbinom{k}{2} d^2$ and the $o(1)$ is uniform over the choice of the classes.
For any graph $G\in \cG_{kd}$, by the previous observation and the definition of the ASP, $\eta_{k,d}(G)$ is equal to
\begin{align*}
\sum_{A_1,\ldots,A_k\in \cC_d}\, \prod_{i=1}^k \bar\nu_d(A_i) \sum_{\substack{G_1\in A_1,\ldots,G_k\in A_k\\ \text{edge-disjoint and}\\ G_1\cup \dotsb\cup G_k = G}}\, \frac{1}{(\e^{-D/2}+o(1))\prod_{i=1}^k |A_i|}
&= \sum_{A_1,\ldots,A_k\in \cC_d}\, \sum_{\substack{G_1\in A_1,\ldots,G_k\in A_k\\ \text{edge-disjoint and}\\ G_1\cup \dotsb\cup G_k = G}}\, \frac{\prod_{i=1}^k \nu_d(G_i)}{{\e^{-D/2}+o(1)}}\\
&= \sum_{\substack{G_1,\ldots,G_k\in \cG_d,\\ \text{edge-disjoint and}\\ G_1\cup \dotsb\cup G_k = G}}\, \frac{\prod_{i=1}^k \nu_d(G_i)}{{\e^{-D/2}+o(1)}},
\end{align*}
where the first equality used that $\nu_d(G_i) |A_i| = \bar\nu_d(A_i)$ for every $i\in [k]$ and the $o(1)$ is uniform over the choice of $G$.
However, the last term is also proportional to $(\nu_d\oplus\dotsb\oplus\nu_d)(G) = \nu_{kd}(G)$ with $\nu_d$ repeated $k$ times, 
where the equality follows from the associativity of the $\oplus$ operation (\Cref{lem:assoc}).
Since each of $\eta_{k,d}$ and $\nu_{kd}$ is a probability distribution, these two distributions must coincide up to the multiplicative $(1+o(1))$-factor, as desired.
\end{proof}

The next lemma couples the distribution of the variables $(Y_i)_{i=1}^{\infty}$ from the definition of the ASP with the distribution $\bar\mu_d$ so that, for every suitably large $k$ and a random variable $X$ with distribution $\bar\mu_d$, we have $X\in \{Y_1,\ldots,Y_k\}$ with probability close to 1.

\begin{lemma}\label{lem:coupleXY}
For every $\eps > 0$ and $d\ge 1$, there are $k = k(\eps,d)\ge 1$ and $n_0 = n_0(\eps,d,k)\ge 1$ with the following property: for every $n\ge n_0$, there is a coupling of the random variable $X$ with distribution $\bar\mu_d$ and the random variables $Y_1,\ldots,Y_k$ from \Cref{def:procedure-sample} such that $X\in \{Y_1,\ldots,Y_k\}$ with probability at least $1-\eps$.
\end{lemma}
\begin{proof}
Fix $k\ge 1$ to be chosen suitably large in the sequel.
We inductively couple $X$ and the sequences $Y_1,\ldots,Y_i$ for every $i\ge 1$ by using a sequence of Bernoulli random variables $I_0=0,I_1,I_2,\ldots$.
For every $i\ge 1$, the event $I_i = 1$ will indicate that $X\in \{Y_1,\ldots,Y_i\}$: in particular, showing that $I_k=1$ holds for some $k\ge 1$ with probability at least $1-\eps$ is enough to conclude.

We iteratively construct a sequence of measures $\zeta_i = \zeta_{i,k,d,n}$ on $\cC_d$ and partition functions 
\[Z_i = Z_{i,k,d,n} = \sum_{A\in \cC_d} \zeta_i(A)\]
for all $i\ge 0$, starting with $\zeta_0 = \bar\mu_d$ and $Z_0 = 1$. 
Fix $i\ge 1$ and suppose that $X$ and the sequence $Y_1,\ldots,Y_{i-1}$ have been coupled, the case $i=1$ being trivial. 
If $I_{i-1}=1$, we set $I_i = 1$, $\zeta_i = \zeta_{i-1}$ and proceed to the next step.
If $I_{i-1}=0$, suppose that the distribution of $X$ conditionally on $I_{i-1}=0$ is equal to $\zeta_{i-1}$; note that this assumption holds when $i=1$.
Then, by \Cref{lem:max_coupling}, we fix a coupling of $X$ conditionally on $I_{i-1}$ and $Y_i$ so that $X=Y_i$ with probability $1-\dtv(\zeta_{i-1}/Z_{i-1},\bar\nu_d)$. 
As mentioned earlier, we let $I_i$ denote the indicator random variable of the event $X=Y_i$
and, for every $S\in \cC_d$, we set
\begin{equation}\label{eq:zeta}
\zeta_i(S) = \frac{\zeta_{i-1}(S)}{Z_{i-1}} - \min\bigg(\frac{\zeta_{i-1}(S)}{Z_{i-1}}, \bar\nu_d(S)\bigg)\qquad \text{and}\qquad Z_i = \sum_{A\in \cC_d} \zeta_i(A).
\end{equation}
By \Cref{lem:max_coupling}(iii) applied with $\mu = \zeta_{i-1}/Z_{i-1}$ and $\nu = \bar\nu_d$, it follows that $\zeta_i/Z_i$ is the distribution of $X$ conditionally on $I_i = 0$ and, moreover, conditionally on $I_{i-1} = 0$, the event $I_i = 0$ is independent of the random variables $Y_1,\ldots,Y_{i-1}$.
Furthermore, note that $Z_i = \mathbb P(I_i = 0\mid I_{i-1}=0) = \dtv(\zeta_{i-1}/Z_{i-1},\bar\nu_d)$ and, as a consequence, for every $k\ge 1$, the probability that $I_k = 0$ is equal to the product $Z_1\cdots Z_k$ by the chain rule. The remainder of the proof is dedicated to showing that this product tends to 0 as $k\to \infty$.

We argue by contradiction. Suppose that $I_i = 0$ and the (decreasing) sequence $(Z_1\dotsb Z_i)_{i\ge 1}$ converges to a strictly positive limit $\alpha > 0$. Fix $k\ge 1$ and, for every $i\ge 0$, denote by $B_i\subseteq \cC_d$ the subset of classes $S$ where $\zeta_i(S) > 0$ and note that the sequence of sets $(B_i)_{i\ge 0}$ is decreasing with respect to inclusion.
We show by induction that, for each class $S\in B_k$ and for every $i\in \{k-1,k-2,\ldots,0\}$, 
\[\zeta_i(S) \ge \bar\nu_d(S) \sum_{j=i}^{k-1} \prod_{\ell = i}^j Z_\ell.\]
The case $i=k-1$ requires that $\zeta_{k-1}(S)\ge \bar\nu_d(S) Z_{k-1}$, which follows from the assumption $S\in B_k$ and the definition of $\zeta_k$ in~\eqref{eq:zeta}.
Suppose that the statement holds for $i+1\in [k-1]$.
Since $S\in B_k\subseteq B_{i+1}$, we know that $\zeta_{i+1}(S)>0$ and, therefore, $\min(\zeta_i(S)/Z_i, \bar{\nu}_d(S)) = \bar{\nu}_d(S)$.
Then, by using~\eqref{eq:zeta} again, we deduce that
\[\zeta_{i+1}(S) = \frac{\zeta_i(S)}{Z_i} - \bar\nu_d(S)
\ge \bar\nu_d(S) \sum_{j=i+1}^{k-1} \prod_{\ell = i+1}^j Z_\ell 
\quad\text{and therefore, }\quad \frac{\zeta_i(S)}{Z_i}\ge  \bar\nu_d(S) + \bar\nu_d(S) \sum_{j=i+1}^{k-1} \prod_{\ell = i+1}^j Z_\ell,\]
implying the induction hypothesis for $i$.

By the induction hypothesis for $i=0$, for every $S\in B_k$, we have that
\[\bar\mu_d(S) = \zeta_0(S)\ge \bar\nu_d(S) \sum_{j=0}^{k-1}\prod_{\ell=0}^j Z_\ell\ge \bar\nu_d(S)\cdot \alpha k.\]
In particular, $\bar\mu_d(B_k)\ge \bar\nu_d(B_k)\cdot \alpha k$ and, as a consequence, 
\[1\ge \mu_d(\{G: \exists S\in B_k, G\in S\})\ge \nu_d(\{G: \exists S\in B_k, G\in S\})\cdot \alpha k.\]
In particular, by \Cref{prop:cont-dtv}(i) and the contiguity of $\mu_d$ and $\nu_d$ justified by \Cref{thm:const}, 
there exists a large enough $k$ so that 
\[\nu_d(\{G: \exists S\in B_k, G\in S\})\le (\alpha k)^{-1}\qquad \text{implies that}\qquad \mu_d(\{G: \exists S\in B_k, G\in S\}) < \alpha.\]
Furthermore, note that, under the event $I_k = 0$, $X$ must belong to a class in $B_k$ with probability 1. Thus,
\[\mu_d(\{G: \exists S\in B_k, G\in S\}) \ge \mathbb P(I_k = 0) = Z_1\dotsb Z_k \ge \alpha,\]
leading to a contradiction.
\end{proof}

Finally, we show that \Cref{lem:coupleXY} implies that $G(n,d)$ and $G(n,1)\oplus\dotsb\oplus G(n,1)$ repeated $kd$ times can be coupled so that the former random graph is included in the latter with probability close to 1.

\begin{lemma}\label{lem:complete_couple}
For every $\eps > 0$ and $d\ge 1$, there exists $k = k(\eps, d)\ge 1$ with the following property: there is a coupling of the random graphs $G$ with distribution $\mu_d$ and 
$G_{\oplus}$ with distribution $\nu_{kd}$ so that $G\subseteq G_{\oplus}$ with probability at least $1-\eps$.
\end{lemma}
\begin{proof}
First of all, fix $k = k(\eps/2,d)$ as given in \Cref{lem:coupleXY} and sample the corresponding (coupled) random variables $X,Y_1,\ldots,Y_k$. 
Next, we sample (according to a coupling to be explained shortly) the random graph $G$ uniformly at random in the set $X$, and the random graph $G_{\oplus}' = H_1\oplus\dotsb\oplus H_k$ according to the ASP.
By definition of $X$, the graph $G$ is sampled uniformly at random in $\cG_d$ while the graph $G_{\oplus}'$ has distribution $\eta_{k,d}$.

By \Cref{lem:coupleXY}, with probability at least $1-\eps/2$, there exists $i\in [k]$ such that $X = Y_i$.
Moreover, observe that, conditionally on the sets $Y_1,\ldots,Y_k$, for every $i\in [k]$, the graph $H_i$ is distributed uniformly on the set $Y_i$: indeed, by symmetry, every graph in $Y_i$ participates in the same number of $k$-tuples of edge-disjoint graphs in $Y_1\times \dotsb \times Y_k$.
As a result, conditionally on the sets $X,Y_1,\ldots,Y_k$ and the event $X=Y_i$ for some $i\in [k]$, the graphs $G$ and $H_i$ can be coupled so that $G=H_i$ with probability 1. Since the probability measures $\nu_{kd}$ and $\eta_{k,d}$ coincide up to a $(1+o(1))$-factor by \Cref{lem:coincide}, one can couple the graphs $G_{\oplus}$ and $G_{\oplus}'$ so that $G_{\oplus} = G_{\oplus}'$ with probability $1-o(1)$, which is enough to conclude.
\end{proof}

\section{\texorpdfstring{Proof of \Cref{thm:Coupling-Inclusion}}{Proof of  Theorem \ref{thm:Coupling-Inclusion}}}\label{sec:proofs}

We are ready to prove \Cref{thm:Coupling-Inclusion}.

\begin{proof}[Proof of \Cref{thm:Coupling-Inclusion}.]
Recall that, by \cite[Theorem~6]{Gao23}, \Cref{thm:Coupling-Inclusion} holds when $\omega((\log n)^7) = d_1\le d_2\le n-1$.
Therefore, it is enough to show the statement when $d_1\le d_2\le z(n) := (\log n)^8$: indeed, for every $d_1\le z(n)$ and $d_2\in [z(n), n-1]$, a coupling satisfying whp $G(n,d_1)\subseteq G(n,d_2)$ can be constructed via two couplings where whp $G(n,d_1)\subseteq G(n,z(n))$ and $G(n,z(n))\subseteq G(n,d_2)$, respectively.
We thus focus on the regime $d_2 \leq z(n)$. 
We consider the cases $d_1 = O(1)$ and $d_1 = \omega(1)$ separately.

\vspace{0.5em}
\noindent
\textbf{Case 1: $d_1 = O(1)$.}
First, we set $d' = d_2/2$ and show that $\dtv(\mu_{d'}\oplus \nu_{d_2-d'}, \mu_{d_2}) = o(1)$. By using the triangle inequality and \Cref{cor:cont_Gnd}(a), we get
\begin{equation}\label{eq:dtv1}
\begin{split}
\dtv(\mu_{d'}\oplus \nu_{d_2-d'},\mu_{d_2})
&\le \sum_{i=0}^{d_2-d'-1} \dtv(\mu_{d'+i}\oplus \nu_{d_2-d'-i}, \mu_{d'+i+1}\oplus \nu_{d_2-d'-i-1})\\
&= \sum_{i=0}^{d_2-d'-1} \dtv((\mu_{d'+i}\oplus\mu_1 )\oplus\nu_{d_2-d'-i-1}, \mu_{d'+i+1}\oplus \nu_{d_2-d'-i-1})\\
&\leq C\cdot \sum_{i=0}^{d_2-d'-1}\dtv(\mu_{d'+i}\oplus \mu_1, \mu_{d'+i+1})\,,
\end{split}
\end{equation}
where $C > 0$ is the absolute constant in \Cref{cor:cont_Gnd}(a).
Then, by \Cref{cor:dTV},
\begin{equation}\label{eq:d'}
\dtv(\mu_{d'}\oplus \eta_{d_2-d'},\mu_{d_2}) \leq  C\cdot \sum_{i=0}^{d_2-d'-1} (d'+i)^{-1.1} = C\cdot \sum_{j=d'}^{d_2-1}j^{-1.1} = o(1)\,.    
\end{equation}
Fix any set of $kd_1$ edge-disjoint perfect matchings $M_1,\ldots,M_{kd_1}\subseteq K_n$.
On the one hand, we show that, for any choice of $M_1,\ldots,M_{kd_1}$, the probability that $M_1,\ldots,M_{kd_1}$ are the first $kd_1$ matchings in the composition of $G(n,1)\oplus\dotsb\oplus G(n,1)\oplus G(n,d')$ with $G(n,1)$ repeated $d_2-d'$ times is the same for all choices of $M_1,\ldots,M_{kd_1}$ up to a multiplicative factor of $(1+o(1))$. 
Indeed, similarly to the proof of \Cref{cor:cont_Gnd}, the latter statement follows by $d'-d_1+1$ consecutive applications of \Cref{thm:McKay}, $d'-d_1$ of them with parameters 
\[x_1=\dotsb=x_n\in [kd_1, d'-1]\qquad \text{and}\qquad g_1=\dotsb=g_n=1\]
and the last one with parameters
\[x_1=\dotsb=x_n=d_2-d'\qquad \text{and}\qquad g_1=\dotsb=g_n=d'.\]
\noindent
As a result, for every $k\ge 1$, there is a coupling ensuring that whp $G(n,1)\oplus\dotsb\oplus G(n,1)$ repeated $kd_1$ times is a subgraph of $G(n,1)\oplus\dotsb\oplus G(n,1)\oplus G(n,d')$ with $G(n,1)$ repeated $d_2-d'$ times.

On the other hand, by \Cref{lem:complete_couple} applied with $d = d_1$, for every $\eps > 0$, there is $k = k(\eps, d_1)\ge 1$ and a coupling where $G(n,d_1)$ is realised as a subgraph of $G(n,1)\oplus\dotsb\oplus G(n,1)$ repeated $kd_1$ times with probability at least $1-\eps$.
By composing the above two couplings and using~\eqref{eq:d'} and the fact that the latter result holds for every $\eps > 0$, we obtain the desired conclusion when $d_1 = O(1)$.

\vspace{0.5em}
\noindent
\textbf{Case 2: $d_1=\omega(1)$.}
By replacing $d'$ with $d_1$ in~\eqref{eq:dtv1} and~\eqref{eq:d'}, one similarly derives that
\begin{equation}\label{eq:d1}
\dtv(\mu_{d_1}\oplus \eta_{d_2-d_1},\mu_{d_2}) \leq  C\cdot \sum_{i=0}^{d_2-d_1-1} (d_1+i)^{-1.1} = C\cdot \sum_{j=d_1}^{d_2-1}j^{-1.1} = o(1)\,.    
\end{equation}
Once again, by $d_2-d_1$ consecutive applications of \Cref{thm:McKay} with 
\[x_1=\dotsb=x_n\in [d_1, d_2-1]\qquad \text{and}\qquad g_1=\dotsb=g_n=1,\] 
for every $d_1$-regular graph $G\subseteq K_n$, the probability that $G$ is the $d_1$-regular graph in the composition of $G(n,d_1)\oplus G(n,1)\oplus\dotsb\oplus G(n,1)$ with $G(n,1)$ repeated $d_2-d_1$ times is the same for all $G$ up to a multiplicative factor of $(1+o(1))$. 
In particular, there is a coupling between the first marginal mentioned above and the random graph $G(n,d_1)$ which identifies the two whp.
By combining this observation with \eqref{eq:d1}, we conclude there exists a coupling such that $\Prob(G_1\subseteq G_2) = 1-o(1)$ where $G_1,G_2$ are sampled from $\mu_{d_1}$ and $\mu_{d_2}$ respectively, as desired.
\end{proof}

\section{Concluding remarks}\label{sec:conclusion}
In this paper, we exhibited several results related to the monotonicity and the nesting properties of random regular graphs, making progress towards \Cref{conj:GIM-Inclusion} by Gao, Isaev and McKay~\cite{GIM22}, and \Cref{conj:IMSZ-coupling-sum} by Isaev, McKay, Southwell and Zhukovskii~\cite{IMcKSZ23}.
While partial progress has been obtained on each of them, completing the general picture remains an important open problem.
One particular point of interest is the converse implication of \Cref{thm:Contiguity-Disjoint-Matchings} for $d\le n^{1/10}$, or more generally for $d\in [n-1]$.
While our proof fails to establish contiguity between $G(n,d)$ and $G(n,1)\oplus\dotsb\oplus G(n,1)$ repeated $d$ times, this would follow directly by combining our approach and a version of \Cref{thm:const} for $k=d$ and $d_1=\dotsb=d_k=1$ with arbitrarily slowly growing $d=\omega(1)$.

It would also be interesting to extend \Cref{thm:Coupling-Inclusion,thm:Coupling-Sprinkling-2-Random} for odd $n$ (as long as the random regular graphs with those parameters exist). 
Again, our approach could be instrumental in deriving a similar bound to \Cref{cor:dTV} for $\dtv(\mu_d \oplus \mu_2, \mu_{d+2})$ but important estimates for the count of 2-factors in random regular graphs are unavailable.
A key missing ingredient is exhibiting good concentration for the number of these $2$-factors: to the best of our knowledge, no such result currently exists in the literature for regular graphs with growing degree (see~\cite{Rob96} for results concerning random regular graphs of bounded degree).

We hope that the results in this work could serve as a valuable tool in the study of random regular graphs. In a forthcoming work~\cite{HLMPW25+}, the authors employ \Cref{thm:Contiguity-Disjoint-Matchings} to establish an approximate version of the Itai-Zehavi conjecture for random regular graphs.

\paragraph{Acknowledgements.} We are grateful to B\'ela Bollob\'as and to Matthew Kwan for their helpful comments and suggestions, and to Nicholas Wormald for useful remarks on a preliminary version of this work and pointing us towards the reference~\cite{Rob96}. 
Part of this research was done during a visit of the fourth author to IST Austria. We thank IST Austria for its hospitality.

\bibliographystyle{amsplain_initials_nobysame_nomr}

\end{document}